\documentclass[10pt,twoside,english,reqno,a4paper]{amsproc}

\usepackage{geometry}
\usepackage{listings,graphicx,amsmath,varioref,amscd,amssymb,color,bm,stmaryrd}
\usepackage{mathpazo} 
\usepackage[scaled=.95]{helvet} 
\usepackage{courier}
\usepackage{amssymb}
\usepackage{amsmath}
\usepackage{amsthm}
\usepackage{amscd}
\usepackage{amssymb}
\usepackage{amsfonts}
\usepackage[english]{babel}
\usepackage[latin1]{inputenc}
\usepackage{esint}
\usepackage[titletoc,toc,title]{appendix}
\usepackage{textcomp}

\definecolor{caribbeangreen}{rgb}{0.0, 0.8, 0.6}
\definecolor{darkpastelgreen}{rgb}{0.01, 0.75, 0.24}
\definecolor{green(pigment)}{rgb}{0.0, 0.65, 0.31}
\usepackage[dvipsnames]{xcolor}
\usepackage{hyperref} 
\hypersetup{
colorlinks=true,
linkcolor= red!70!black,
citecolor= orange,
urlcolor=Sepia, 
}

\usepackage[titletoc,toc,title]{appendix}

\usepackage{graphics}
\usepackage{pgfplots}
\pgfplotsset{width=7cm,compat=newest} 
\usepackage{graphics}
\usepackage{tikz}
\usetikzlibrary{shapes.geometric}
\usepackage{tikz-cd}
\usepackage{caption}
\usepackage{subcaption}

\usepackage{float}

\newtheoremstyle{note}
{3pt}
{3pt}
{}
{}
{\slshape}
{:}
{.5em}
{}
\newtheorem{theorem}{Theorem}[section]
\newtheorem{theo}{Theorem}
\newtheorem{proposition}[theorem]{Proposition}

\newtheorem{corollary}[theorem]{Corollary}

\newtheorem{definition}[theorem]{Definition}
\newtheorem{remark}[theorem]{Remark}

\renewcommand{\theequation}{\thesection.\arabic{equation}}

\long\def\salta#1{\relax}

\include{MACRO.tex}
\def\dive {\text{div }}
\def\N{\nabla}
\def\vp{\varphi}
\newcommand{\integrale}{\int_\Omega}
\def\al{\alpha}
\def\vare{\varepsilon}
\def\ro{\rho}
\newcommand{\meas}{\text{meas}}
\newcommand{\lm}{\lambda}
\newcommand{\ga}{\gamma}
\newcommand{\si}{\sigma}
\newcommand{\ds}{\displaystyle}
\newcommand{\D}{\Delta}

\newcommand{\om}{\omega}
\definecolor{seagreen}{rgb}{0.18, 0.55, 0.34}

\def\og{\leavevmode\raise.3ex\hbox{$\scriptscriptstyle\langle\!\langle$~}}
\def\fg{\leavevmode\raise.3ex\hbox{~$\!\scriptscriptstyle\,\rangle\!\rangle$}}

\newtheorem{theoa}{Theorem}

\newtheorem{propa}[theoa]{Proposition}

\newtheorem{lemb}{Lemma}

\newtheorem{rmkappb}[lemb]{Remark}

\newtheorem{theoc}{Theorem}

\usepackage{hyperref}
\hypersetup{
colorlinks=true,
linkcolor=red!70!black, 
citecolor= orange, 
urlcolor=blue, 
}

\usepackage[capitalize,nameinlink,noabbrev,nosort]{cleveref} 

\crefname{section}{Section}{Sections}
\crefname{subsection}{Subsection}{Subsections}
\crefname{appsec}{Appendix}{Appendices}
\crefname{appendix}{Appendix}{}
\crefname{equation}{}{}
\crefname{figure}{Figure}{Figures}
\crefname{theo}{Theorem}{Theorems}

\crefname{definition}{Definition}{Definitions}
\crefname{theorem}{Theorem}{Theorems}
\crefname{proposition}{Proposition}{Propositions}
\crefname{corollary}{Corollary}{Corollaries}
\crefname{remark}{Remark}{Remarks}
\crefname{lemma}{Lemma}{Lemmas}

\crefname{theoa}{Theorem}{Theorems}
\crefname{propa}{Proposition}{Propositions}
\crefname{rmkappa}{Remark}{Remarks}
\crefname{lema}{Lemma}{Lemmas}

\author[M. Magliocca]{Martina Magliocca}

\address[M. Magliocca]{Dipartimento di Matematica, Universit\`a degli Studi Tor Vergata, Via della Ricerca Scientifica 1, 00133 Rome, Italy. 
	\\ \url{magliocc@mat.uniroma2.it}}

\keywords{Nonlinear parabolic equations, Unbounded data, Repulsive Gradient} \subjclass[2000]{35K55,35K61,35R05}

\begin{document}

\title[Existence results for parabolic problem with a repulsive gradient term]{Existence results for a Cauchy-Dirichlet parabolic problem\\ with a repulsive gradient term}

\begin{abstract}
We study the existence of solutions of a nonlinear parabolic problem of Cauchy-Dirichlet type having a lower order term which depends on the gradient. The model we have in mind is the following:
\begin{equation*}
\begin{cases}
\begin{array}{ll}
 u_t-\dive (A(t,x)\N u|\N u|^{p-2})=\gamma |\N u|^q+f(t,x) &\text{in } Q_T,\\
 u=0  &\text{on }(0,T)\times \partial \Omega,\\
 u(0,x)=u_0(x) &\text{in } \Omega,
\end{array}
\end{cases}
\end{equation*}
where $Q_T=(0,T)\times \Omega$, $\Omega$ is a bounded domain of $\mathrm{R}^N$, $N\ge 2$, $1<p<N$, the matrix $A(t,x)$ is coercive and with {measurable bounded coefficients}, the r.h.s. growth rate satisfies the{ superlinearity condition}
\[
\max\left\{\frac{p}{2},\frac{p(N+1)-N}{N+2}\right\}<q<p
\]
and the initial datum $u_0$ is an unbounded function belonging to a suitable {Lebesgue space} $L^\si(\Omega)$. We point out that, once we have fixed $q$, there exists a link between this growth rate and exponent $\si=\si(q,N,p)$ which allows one to have (or not) an existence result.  Moreover, the value of $q$ deeply influences the notion of solution we can ask for.

The {sublinear growth} case with
\[
0<q\le\frac{p}{2}
\]
is dealt at the end of the paper for what concerns small value of $p$, namely $1<p<2$.
\end{abstract}

\maketitle
\tableofcontents

\section{Introduction}

\setcounter{equation}{0}

Let $Q_T=(0,T)\times \Omega$, where $\Omega$ is a bounded domain of $\mathrm{R}^N$, $N\ge 2$.\\ 
We are interested in the study of existence results concerning the following nonlinear parabolic problem of Cauchy-Dirichlet type:
\begin{equation}\label{pb}
\begin{cases}
\displaystyle
\begin{array}{ll}
u_t-\dive a(t,x,u,\nabla u)=H(t,x,\nabla u)&\text{in}\,\, Q_T,\\
u=0  &\text{on}\,\,(0,T)\times \partial \Omega,\\
u(0,x)=u_0(x) &\text{in}\,\, \Omega,
\end{array}
\end{cases}
\end{equation}
where the initial datum $u_0=u_0(x)$ is a possibly \emph{unbounded} function belonging to a suitable {Lebesgue space} $L^\sigma(\Omega)$, the operator $-\dive a(t,x,u,\nabla u)$ satisfies conditions of Leray-Lions type in the space $L^p(0,T;W^{1,p}_0(\Omega))$ with $1< p<N$, the r.h.s. $H(t,x,\nabla u)$ is supposed to grow at most as \emph{a power of the gradient} plus a {forcing term}, namely $\ds|H(t,x,\nabla u)|\le \gamma |\N u|^q+f$, $\gamma>0$, provided that $f=f(t,x)$ belongs to a suitable space $L^r(0,T;L^m(\Omega))$ and the gradient growth rate is such that $\ds q<p$.

The model equation one has to keep in mind is the following:
\begin{equation}\label{Dp}
u_t-\D_pu=\gamma|\N u|^q+f\quad\text{in}\,\, Q_T 
\end{equation}
where $\D_p v$ is the $p$-Laplace operator, namely $\D_pv=\dive (|\N v|^{p-2}\N v)$.\\

We give a very brief recall aimed at motivating both the {mathematical} and {physical} interest in the study of problem \eqref{Dp}. Consider, for the sake of simplicity, the linear case $p=2$ and thus the equation we take into account is
\begin{equation}\label{es1} 
\displaystyle{ u_t-\Delta u=|\N u|^q+f(t,x)\quad\text{in}\,\, Q_T.}
\end{equation}
The equation \eqref{es1} can be seen, up to scaling, as the approximation in the viscous sense ($\vare\to 0^+$) of Hamilton-Jacobi equations. We refer to \cite{L} for a deeper analysis in this sense. 
Moreover, \eqref{es1} is studied in the physical theory of growth and roughening of surfaces as well and it is known under the name of Kardar-Parisi-Zhang equation (see, for instance, \cite{KPZ,KS}). We refer to \cite{SZ} for a detailed overview on the several applications of \eqref{es1}. Finally, equation \eqref{Dp} is the simplest model for a quasilinear second order parabolic problem with nonlinear reaction terms of first order.\\

Here we list some previous papers and results to explain what is known in the literature.

The case $p=2$ with \emph{Laplace operator}, $f=0$ and unbounded initial data belonging to Lebesgue spaces has been extensively studied in \cite{BASW}. The authors provide a detailed investigation of the Cauchy problem 
\begin{equation}\label{hj}
\begin{cases}
\begin{array}{ll}
u_t-\Delta u=\gamma |\N u|^q & \text{in}\,\,(0,T)\times \mathbb{R}^N,\\
u(0,x)=u_0(x) & \text{in}\,\,\mathbb{R}^N
\end{array}
\end{cases}
\end{equation}
assuming that $q>1$ and $\gamma\in \mathbb{R}$, $\gamma\ne 0$. Their approach to the study of \eqref{hj} goes through {semigroup theory} and {heat kernel estimates} and points out that one is allowed to have (or not) existence of a solution $u$ only if the gradient growth $q$ and the integrability class of $u_0$ satisfy a {precise relation}.
To be clear, they show that, for fixed value of $2-\frac{N}{N+1}<q<2$, $u_0$ has to be taken in the Lebesgue space $L^\sigma(\Omega)$ for $\ds\sigma=\frac{N(q-1)}{2-q}$ while, if $q<2-\frac{N}{N+1}$, data measures are allowed. 
Nonexistence and nonuniqueness results are also proved for positive data $u_0\ge 0$ whereas $\gamma>0$, $q<2$ and  $u_0\in L^{\si}(\Omega)$ for $\sigma<\frac{N(q-1)}{2-q}$. 
In addition, the authors take into account initial data in Sobolev's spaces, as well as the cases of attractive gradient ($\gamma<0$) with positive initial data and of supernatural growth $q\ge 2$ with $\sigma\ge 1$ (in which existence fails).\\
Even if this reference is concerned with the Cauchy problem, several arguments are actually local in space.

In a similar spirit, we refer to \cite{BD} for the study of the Cauchy-Dirichlet problem in the case of \emph{Laplace operator}, $f=0$ and $q>0$, namely
\begin{equation*}
\begin{cases}
\begin{array}{ll}
u_t-\Delta u=\gamma |\N u|^q & \text{in}\,\,(0,T)\times \Omega,\\
u=0 & \text{on}\,\, (0,T)\times\partial \Omega,\\
u(0,x)=u_0(x) & \text{in}\,\,\Omega.
\end{array}
\end{cases}
\end{equation*}
The authors provide existence, uniqueness and regularity results when the nature of the gradient is both repulsive and attractive and the initial datum is a function belonging to a Lebesgue space or a bounded Radon measure as well. The same problem is dealt with in \cite{BDL} for regular initial data $u_0\in C_0(\overline{\Omega})$ for what concerns the long time behaviour of the solution. 

We underline that, because of the semigroup theory approach and the heat kernel regularity, the results proved in the works just mentioned cannot be extended to problems  with more general operators like those considered in this paper, which include nonlinear operators with \emph{measurable coefficients}.

As for the case of $p$-Laplace operator with $p>2$, problem \eqref{Dp} was treated in \cite{At} for a gradient growth rate satisfying $q>p-1$, $f=0$ and $u_0\in W^{1,\infty}(\Omega)$.

As far as general operators in divergence form are concerned, previous works have considered either the case that $q=p$ or the case that $q$ is sufficiently small.

We refer to \cite{BMP,OP,DGP,DGL} for what concerns the case in which the r.h.s. has {natural} growth (i.e. $q=p$). In particular, the problem \eqref{pb} with $q=p$ is studied in \cite{DGP} while \cite{DGL} generalizes the r.h.s. to $\beta(u)|\N u|^p+f$, $\beta(\cdot)$ bounded, polynomial or exponential as well.\\
Furthermore, in \cite{Po,DNFG} the authors take into account the case where the growth of the gradient is determined by the value
\begin{equation}\label{g}
q=\frac{p(N+1)-N}{N+2}
\end{equation}
which corresponds to a ''linear'' growth, in a sense specified later.
We point out that \cite{Po} analyses the existence of weak solutions belonging at least to $L^1(0,T;W^{1,1}_0(\Omega))$ and thus a lower bound for $p$ is required.\\
However, the critical growth in \eqref{g} seems not to be always sharp for the problem to exhibit a ''linear'' behaviour. We will show later some arguments aimed to justify our claim.\\

Our purpose is filling the gap between the cases with ''sublinear'' and natural growth for what concerns general operators in divergence form. This means that we deal with problems which have ''superlinear'' growth in the gradient term (we will explain and comment later what we mean for ''superlinearity''). To some extent, this work is the extension to parabolic equations of similar results obtained in \cite{GMP} for stationary problems.

\subsection{Comments on the $q$ growth and comparison with the stationary problem}

We here present the stationary case of the problem \eqref{pb} which is deeply analysed in \cite{GMP,FM} (see also \cite{AFM}). Such a problem reads
\begin{equation*}
\begin{cases}
\begin{array}{ll}
\alpha_0 u-\dive a(x,u,\nabla u)=H(x,\nabla u) &\text{in}\,\, \Omega,\\ u=0  &\text{on}\,\, \partial \Omega
\end{array}
\end{cases}
\end{equation*}
and the hypotheses assumed are the following: $\alpha_0\ge0$, the operator $-\dive (a(x,u,\N u))$ satisfies conditions of Leray-Lions type, the r.h.s. growth is determined by $|H(x,\xi)|\le \gamma|\xi|^q+f(x)$ for $\ds p-1<q<p$ and 
\[
f\in L^m(\Omega)\,\,\text{where}\,\,m=\frac{N(q-(p-1))}{q}
\]
provided that $m\ge 1$. We point out that the particular value of $m$ above is optimal in the sense that it represents the minimal regularity one has to require on the source term $f\in L^m(\Omega)$ in order to have an existence result. We also refer to \cite{HMV} for further comments on this sense. \\
The a priori estimates proved in \cite{GMP} state that
\begin{equation}\label{sapell}
\|\N[(1+|u|)^\rho]\|_{L^p(\Omega)}\le M
\end{equation}
where $\rho=\frac{(N-p)(q-(p-1))}{p(p-q)}$ and with the constant $M$ depends on the parameters of the problem and, above all, on the forcing term. The dependence on $f$ varies if $\alpha_0=0$ or $\alpha_0>0$. More precisely, if $\alpha_0=0$ a size condition on the data is required. On the contrary, the case $\alpha_0>0$ (which is the closest to the parabolic problem) does not need such a condition. In particular, in this last case, $M$ remains bounded when $f$ \emph{varies in sets which are bounded and equi-integrable in} $L^m(\Omega)$.
We underline that such a kind of dependence on the datum is due to the fact that we are in the \emph{superlinear growth setting}.\\
Roughly speaking, the l.h.s. grows like a $(p-1)$-power of $|\N u|$ and thus we are saying that the r.h.s. grows faster (indeed $q>p-1$).

We conclude by pointing out that, on account of \eqref{sapell}, depending on $\ro\ge 1$ or $\ro<1$ we have different ranges of $q$ which lead to either solutions of finite energy or solutions with infinite energy.\\

We will prove later that the parabolic problem \eqref{pb} verifies an estimate which is similar to the one in \eqref{sapell}. To be more precise, such an estimate has the form
\begin{equation}
\|u\|_{L^\infty(0,T;L^\si(\Omega))}+\|\N[(1+|u|)^\beta]\|_{L^p(Q_T)}\le M
\end{equation}
where the value of $\si$, depending on $p$ and $q$, will be discussed later and $\beta=\beta(p,\si)$.\\
Concerning the similarity of the estimate, we want to underline that, again, the constant $M$ remains bounded when $u_0$ \emph{and} $f$ \emph{vary in sets which are bounded and equi-integrable in} $L^\si(\Omega)$ and $L^r(0,T;L^m(\Omega))$, for suitable values of $m$ and $r$ which will be largely commented in the following. We come back to emphasize that such a dependence is due to the \emph{''superlinearity'' growth rate} of the r.h.s..

On the other hand, as far as the ''superlinearity'' threshold of the $q$ growth is concerned, we point out that the parabolic setting carries out noteworthy {differences} compared to the elliptic one. Indeed, the presence of the time derivative $u_t$ in \eqref{pb} influences the relation between $q$ and $p$ and this fact clearly does not occur if we deal with stationary equations. We refer to \cite[Remark $3$]{DNFG} for additional comments on this fact. \\

We will explain soon that the threshold between linear/superlinear growth depends on the values of $p$ we are taking into account.

\section{On the superlinear setting}\label{subsecsharp}
In what follows, we are going to motivate the {superlinear thresholds} we will take into account during the paper. Moreover, we will highlight the link between the $q$ growth of the gradient term and the Lebesgue spaces where the data $u_0$ and $f$ have to be taken in order to have an {existence result}.
In order to explain the assumptions we will require later, let us consider the Cauchy-Dirichlet problem for the standard $p$-Laplace operator
\begin{equation}\label{eqf}
\begin{cases}
\begin{array}{ll}
u_t-\D_p u=f  & \text{in}\,\,Q_T,\\
 u=0  &\text{on}\,\,(0,T)\times\partial \Omega,\\
 u(0,x)=0  &\text{in} \,\,\Omega,
\end{array}
\end{cases}
\end{equation}
where $f\ge 0$ satisfies
\begin{equation*}
f\in L^a(Q_T)\quad\text{with}\quad a\le\left( p \frac{N+2}{N} \right)'.
\end{equation*}
The function $f$ will later play the role of the gradient term $|\N u|^q$. The bound assumed on $a$ allows us to give here a simple explanation through energy estimates. Indeed, for more regular forcing terms a similar explanation should make use of potential theory and Calderon-Zygmund estimates. Since the whole work will turn around energy estimates, it seems better to present a consistent argument below. In any case, the full range of $a$ will be dealt with later in the article. We also restrict the present discussion to the case $u_0=0$ (thus $u\ge 0$) for simplicity. 

Basically, we look for a $L^1(Q_T)$ estimate of a suitable power of the gradient in term of the forcing terms $f$ of the form
\[
\||\N u|^b\|_{L^1(Q_T)}\le c\|f\|_{L^a(Q_T)}^\ro.
\]
Then, when $f=|\N u|^q$, we wonder if 
\[
\||\N u|^b\|_{L^1(Q_T)}\le c\||\N u|^q\|_{L^{a}(Q_T)}^{\ro}
\]
provides a useful estimate and in this case whether  the estimate has a ''sublinear'' or ''superlinear'' character. The first question leads us to the condition
\begin{equation}\label{cond1}
aq\le b
\end{equation}
in order to close the estimate.\\
Then, the ''superlinear'' homogeneity of the estimate holds if
\begin{equation}\label{cond2}
\frac{\ro q}{b}>1.
\end{equation}
In order to find the exponents $b$ and $\ro$ involved, we formally multiply \eqref{eqf} by $\vp(u)=\left( (1+u)^{\nu-1}-1 \right)$ with $\nu=\nu(a)\in (1,2)$ to be fixed. Thus, we have
\begin{equation}\label{ddd}
\integrale \Phi(u(t))\,dx+\iint_{Q_t}\frac{|\N u|^p}{(1+u)^{2-\nu}}\,dx\,ds\le c\iint_{Q_t}f(1+u)^{\nu-1}\,dx\,ds
\end{equation}
with $\Phi(\cdot)$ defined as $\Phi(u)=\int_0^u \left( (1+z)^{\nu-1}-1 \right)\,dz$. An application of H\"older inequality with indices $(a,a')$ and the inequality $c(u^\nu-1)\le 
\Phi(u)$ provide us with the following estimate:
\begin{equation}\label{dis11}
\integrale (u(t))^\nu\,dx+\iint_{Q_t}\frac{|\N u|^p}{(1+u)^{2-\nu}}\,dx\,ds\le c\|f\|_{L^a(Q_t)}\|1+u\|_{L^{a'(\nu-1)}(Q_t)}^{\nu-1}+c.
\end{equation}
We now define $\ds v=(1+u)^{\frac{\nu+p-2}{p}}$ and rewrite \eqref{dis11} in terms of $v$:
\begin{equation}\label{dis4}
\begin{split}
\integrale (v(t))^{\tilde{\nu}}\,dx+\iint_{Q_t}|\N v|^p\,dx\,ds&\le c\|f\|_{L^{a}(Q_t)}\|v\|_{L^{a'\frac{p(\nu-1)}{\nu+p-2}}(Q_t)}^{\frac{p(\nu-1)}{\nu+p-2}}+c\\
\end{split}
\end{equation}
where $\tilde{\nu}=\frac{p\nu}{\nu+p-2}$. Invoking Theorem \ref{teoGN} with $h=\tilde{\nu}$, $\eta=p$ and $w=y$ allows us to deduce that $v\in L^{p\frac{N+\tilde{\nu}}{N}}(Q_t)$. We point out that, in order to apply Gagliardo-Nirenberg inequality, we need $\tilde{\nu}<p^*$: thus, algebraic computations force us to require $\ds p>\frac{2N}{N+\nu}$. We go further imposing $a'(\nu-1)\frac{p}{\nu+p-2}=p\frac{N+\tilde{\nu}}{N}$, otherwise
\begin{equation}\label{nua}
\nu=\nu(a)=N\frac{a(p-1)-(p-2)}{N-p(a-1)}.
\end{equation}
Combining \eqref{disGN=} with \eqref{dis4}, we deduce that
\begin{align*}
\|v\|_{L^{p\frac{N+\tilde{\nu}}{N}}(Q_t)}^{p\frac{N+\tilde{\nu}}{N}}\le 
c\|f\|_{L^a(Q_t)}^{\frac{N+p}{N}}
\|v\|_{L^{p\frac{N+\tilde{\nu}}{N}}(Q_t)}^{p\frac{N+\tilde{\nu}}{N}\frac{N+p}{a'N}}+c
\end{align*}
which, being $\frac{N+p}{a'N}<1$, leads us to the following inequality:
\begin{equation}\label{dis22}
\|v\|_{L^{p\frac{N+\tilde{\nu}}{N}}(Q_t)}^{p\frac{N+\tilde{\nu}}{N}}\le c\|f\|_{L^a(Q_t)}^{\frac{a(N+p)}{N-p(a-1)}}+c.
\end{equation}
We now focus on gradient estimates and consider the gradient power $|\N u|^b$ with $b\le p$. \\
If $a<\left(p\frac{N+2}{N}\right)'$, we let $b<p$ and we multiply $|\N u|^b$ by $(1+u)^{\frac{b(2-\nu)}{p}}(1+u)^{-\frac{b(2-\nu)}{p}}$, so that H\"older's inequality with $\left(\frac{p}{b},\frac{p}{p-b}\right)$ provides us with
\begin{equation}\label{dis3}
\begin{split}
\iint_{Q_t}|\N u|^b\,dx\,ds&\le \left( \iint_{Q_t}\frac{|\N u|^p}{(1+u)^{2-\nu}}\,dx\,ds \right)^{\frac{b}{p}} \left(\iint_{Q_t} (1+u)^{\frac{b(2-\nu)}{p-b}}\,dx\,ds\right)^{\frac{p-b}{p}}\\
&\le c \left( \iint_{Q_t}|\N v|^p\,dx\,ds \right)^{\frac{b}{p}} \left(\iint_{Q_t} v^{\frac{b(2-\nu)}{p-b}\frac{p}{\nu+p-2}}\,dx\,ds\right)^{\frac{p-b}{p}}.
\end{split}
\end{equation}
We thus impose $ \frac{b(2-\nu)}{p-b}\frac{p}{\nu+p-2} =p\frac{N+\tilde{\nu}}{N} $ which, in turn, implies that the exponent $b$ has the following form:
\begin{equation}\label{bnu}
b=b(\nu)=\frac{N(\nu+p-2)+\nu p}{N+\nu}.
\end{equation}
Such a value of $b$, combined with the one previously found for $\nu$ (see \eqref{nua}), becomes
\begin{equation}\label{ba}
b=b(a)=a\frac{p(N+1)-N}{N-a+2}.
\end{equation}
If $a=\left(p\frac{N+2}{N} \right)'$, then we set $b=p$ and we have $\nu=\tilde{\nu}=2$ thanks to the definition of $\tilde{\nu}$ and to \eqref{nua}. Then, \eqref{dis11} implies that
\begin{equation}\label{d33}
\integrale (u(t))^{2}\,dx+\int_{0}^t\|\N u(s)\|_{L^p(\Omega)}^p\,ds\le c\|f\|_{L^{a}(Q_t)}\|1+u\|_{L^{p\frac{N+2}{N}}(Q_t)}+c,
\end{equation}
whether \eqref{dis22} becomes
\begin{equation}\label{d44}
\|u\|_{L^{p\frac{N+2}{N}}(Q_t)}^{p\frac{N+2}{N}}\le c\|f\|_{L^a(Q_t)}^{\frac{p(N+2)(N+p)}{N(p(N+1)-N)}}+c.
\end{equation}
We point out that, since $\nu=2$, then \eqref{bnu} provides us with $b=p$.\\
Note that the assumption $a\le\left(p\frac{N+2}{N} \right)'$ ensures that $b\le p$. \\
The bounds in \eqref{dis4}, \eqref{dis22} and the inequality in \eqref{dis3} give the desired estimate of the gradient we were looking for, namely
\begin{equation*}
\| |\N u|^b\|_{L^1(Q_t)}\le c\|f\|_{L^a(Q_t)}^{\frac{a(N+2)}{N-a+2}}+c=c\|f\|_{L^a(Q_t)}^{\frac{b(N+2)}{p(N+1)-N}}+c
\end{equation*}
where the equality is due to the value of $b$ in \eqref{ba}.\\
Now we let $f=|\N u|^q$, thus our last estimate becomes
\begin{equation*}
\| |\N u|^b\|_{L^1(Q_t)}\le c\||\N u|^q\|_{L^{a}(Q_t)}^{\frac{b(N+2)}{p(N+1)-N}}+c.
\end{equation*} 

We are ready to check the conditions in \eqref{cond1} and \eqref{cond2}. We first require that
\[
aq\le b
\]
which implies that the estimate be closed giving
\begin{equation*}
\| |\N u|^b\|_{L^1(Q_t)}\le c\||\N u|^b\|_{L^{1}(Q_t)}^{\frac{q(N+2)}{p(N+1)-N}}+c.
\end{equation*}
The condition $aq\le b$, combined with \eqref{nua} and \eqref{bnu}, implies that
\begin{equation}\label{ni}
\nu\ge \frac{N(q-(p-1))}{p-q}.
\end{equation}
Moreover, \eqref{nua} and \eqref{ni} lead us to
\[
a\ge\frac{N(q-(p-1))+2q-p}{q}.
\] 
We also notice that the same computation would remain unchanged if we had an initial datum $u_0$ belonging to $L^\si(\Omega)$ with $\ds \si=\frac{N(q-(p-1))}{p-q}$. In particular, we have found the relations between the growth rate $q$ and the summability of $u_0$ and $f$ which are needed in order to have an existence result for \eqref{Dp}. 

As far as the homogeneity of the estimate is concerned, we notice that
\[
q\frac{N+2}{p(N+1)-N}>1
\]
leads us to the following superlinearity threshold for the growth $q$:
\[
q>\frac{N(p-1)+p}{N+2}.
\]

If $1<p\le \frac{2N}{N+\nu}$, we cannot apply Gagliardo-Nirenberg regularity result but we know that, at least, we can impose $\ds a'(\nu-1)=\nu$ in \eqref{dis11} (i.e. $a=\nu$) and $\ds \frac{b(2-\nu)}{p-b}=\nu$ in \eqref{dis3} (i.e. $b=\frac{\nu p}{2}$). We underline that the above conditions can be asked (and thus the estimate below holds) for \emph{every value} of $p$. Then, another gradient estimate is given by
\[
\| |\N u|^{b}\|_{L^1(Q_t)}\le c\|f\|_{L^{\nu}(Q_t)}^{\nu}.
\]
Letting $f=|\N u|^q$ in our last inequality provides that
\begin{equation}\label{sub}
\| |\N u|^{\frac{p}{2}\nu}\|_{L^1(Q_t)}\le c\| |\N u|^q\|_{L^{\nu}(Q_t)}^{\nu}+c.
\end{equation}
Now the estimate is closed whenever we have that 
\[
q\le\frac{p}{2}
\]
and, in this case, we always fall within a sublinear type of estimate. In particular, this means that the r.h.s. of \eqref{Dp} shows a \emph{sublinear} (\emph{linear}) growth for $q<\frac{p}{2}$ ($q=\frac{p}{2}$).\\

We are going to clarify the meaning of the thresholds discovered above.

We first point out that, once we take $\nu=\si=\frac{N(q-(p-1))}{p-q}$, which gives the least (and therefore optimal) integrability condition on the data, then the following double implication holds:
\[
p>\frac{2N}{N+\nu}\quad\Longleftrightarrow\quad q>\frac{p}{2}.
\]
This means that we have to require $q>\max\left\{ \frac{p}{2}, \frac{N(p-1)+p}{N+2} \right\}$ in order to have a superlinear character in the growth of the r.h.s.. We conclude that the \emph{superlinearity thresholds} are
\[
\begin{array}{c}
\ds
q>\frac{p}{2}\qquad\text{if}\qquad 1<p<2,\\
[3mm]
\ds
q>\frac{N(p-1)+p}{N+2}\qquad\text{if}\qquad p\ge 2.
\end{array}
\]
In this range, the superlinear character of the estimate does not allow us to deduce an a priori estimate from the above arguments. The reader will see additional arguments, based on equi-integrability and continuity, in the proof of our result.\\

In the above explanation of the natural thresholds of the problem, we have supposed that the forcing term in \eqref{eqf} fulfils the same regularity in space and time. We now consider the case in which time and spatial summability may be different, i.e. we take $f\in L^{r}(0,T;L^{m}(\Omega))$, and we wonder which is the curve where the exponents $(m,r)$ can live on in order to have an existence result.\\
To this aim, we come back to \eqref{ddd} and, thanks to twice applications of H\"older's inequality with indices $(m,m')$ and $(r,r')$ in the r.h.s., we get
\begin{equation}\label{dis44}
\begin{split}
\integrale (v(t))^{\tilde{\nu}}\,dx+\iint_{Q_t}|\N v|^p\,dx\,ds&\le c\|f\|_{L^{r}(0,T;L^{m}(\Omega))}\|v\|_{L^{\bar{r}}(0,T;L^{\bar{m}}(\Omega))}^{\frac{p(\nu-1)}{\nu+p-2}}+c
\end{split}
\end{equation}
where
\[
\bar{r}=r'(\nu-1)\frac{p}{\nu+p-2}\qquad\text{and}\qquad \bar{m}=m'(\nu-1)\frac{p}{\nu+p-2}.
\]

If $\frac{2N}{N+\nu}<p<N$, we apply again Theorem \ref{teoGN} with $h=\tilde{\nu}$ and $\eta=p$ but we now focus on the case $w\ne y$. Then, we have
\begin{equation}\label{dis2}
\int_0^t \|v(s)\|_{L^{w}(\Omega)}^{y}\,ds\le c\|v\|_{L^\infty(0,T;L^{\tilde{\nu}}(\Omega))}^{y-p}\int_0^t\|\N v(s)\|_{L^p(\Omega)}^p\,ds
\end{equation}
where $(w,y)$ satisfy the relation
\begin{equation}\label{gn}
\frac{N\tilde{\nu}}{w}+\frac{N(p-\tilde{\nu})+p\tilde{\nu}}{y}=N.
\end{equation}
Observe that, if $r\ne m$, then $y\ne w$ and vice versa.\\
We go further requiring $w\ge\bar{m}$ and $y\ge\bar{r}$: in this way, \eqref{gn} leads us to the condition
\begin{equation*}
\frac{N\nu}{m}+\frac{N(p-2)+p\nu}{r}\le N(p-1)+ p\nu
\end{equation*}
i.e. the admissible range of the values $(m,r)$ for initial datum $u_0$ fixed in $L^\nu(\Omega)$.\\
In particular, we have that such a value of $\nu$ fulfils
\begin{equation}\label{nu}
\nu=Nm\frac{r(p-1)-(p-2)}{Nr-pm(r-1)}.
\end{equation}
if $w=\bar{m}$ and $y=\bar{r}$, i.e., when we assume the lowest regularity on $(m,r)$.

Finally, if $1<p\le \frac{2N}{N+\nu}$, the regularity $u\in L^{\infty}(0,T;L^{\nu}(\Omega))$ allows us to take $f\in L^1(0,T;L^{\nu}(\Omega))$ as the best choice. Observe that, letting $\nu=m$ in \eqref{nu}, gives
\[
[p(N+\nu)-2N](r-1)=0
\]
and thus there is continuity between the case $1<p\le \frac{2N}{N+\nu}$ and $\frac{2N}{N+\nu}<p<N$ for what concerns the values $(m,r)$.

\section{Assumptions and statements}\label{secassnot}
\setcounter{equation}{0}

\subsection{Assumptions}
Let us consider the following nonlinear parabolic Cauchy-Dirichlet problem:
\begin{equation}\label{P}\tag{P}
\begin{cases}
\begin{array}{ll}
u_t-\dive  a(t,x,u,\nabla u) =H(t,x,\nabla u) & \text{in}\,\,Q_T,\\ 
 u=0  &\text{on}\,\,(0,T)\times \partial \Omega,\\
 u(0,x)=u_0(x) &\text{in}\,\, \Omega,
\end{array}
\end{cases}
\end{equation}
where $\Omega$ is an open bounded subset of $\mathbb{R}^N$,
$N\ge 2$, $Q_T=(0,T)\times \Omega$,  $a(t,x,u,\xi):(0,T)\times \Omega\times \mathbb{R}\times \mathbb{R}^N\to  \mathbb{R}^N$ and $H(t,x,\xi):(0,T)\times\Omega\times \mathbb{R}^N\to \mathbb{R}$ are Caratheodory functions (i.e. measurable with respect to $(t,x)$ and continuous in $(u,\xi)$).

We assume that the functions $a(t,x,u,\xi)$ and $H(t,x,\xi)$ are such that
\begin{itemize}
\item  the classical Leray-Lions structure conditions hold:
\begin{equation}\label{A1}\tag{A1}
\exists \,\alpha>0:\quad
\alpha|\xi|^p\le a(t,x,u,\xi)\cdot\xi,
\end{equation}
\begin{equation}\label{A2}\tag{A2}
\exists \,\lambda>0:\quad|a(t,x,u,\xi)|\le \lambda[|u|^{p-1}+|\xi|^{p-1}+h(t,x)]\quad\text{where }\,\, h\in L^{p'}(Q_T),
\end{equation}
\begin{equation}\label{A3}\tag{A3}
 (a(t,x,u,\xi)-a(t,x,u,\eta))\cdot(\xi-\eta)> 0 
\end{equation}
with $1<p<N$, for almost every $(t,x)\in Q_T$, for every $u\in \mathbb{R}$ and for every $\xi$, $\eta$ in $\mathbb{R}^N$, $\xi\ne\eta$;
\item the r.h.s. satisfies the growth condition:
\begin{equation}\label{H}\tag{H}
\ds
\begin{array}{c}
\exists \,\,\gamma\,\,\text{s.t. }\,\,
|H(t,x,\xi)|\le \gamma |\xi|^q+f(t,x)\\
[3mm]
\ds
 \text{with}\,\, \max\left\{\frac{p}{2},\frac{p(N+1)-N}{N+2}\right\}<q<p
\end{array}
\end{equation}

for almost every $(t,x)\in Q_T$, for every $\xi$ in $\mathbb{R}^N$ and for some forcing term $f$.
\end{itemize}

Now we detail the hypotheses we make on the data $u_0$ and $f$, taking into account the different ranges of $p$ and $q$. We say that if
\[
1<p<N\quad\text{and}\quad\max\left\{\frac{p}{2},p-\frac{N}{N+1}\right\}<q<p
\]
then we fix
\begin{itemize}
\item the initial datum $u_0$ in the following Lebesgue space:
\begin{equation}\label{ID1}\tag{ID1}
u_0\in L^{\sigma}(\Omega)\quad\text{with}\quad \sigma=\frac{N(q-(p-1))}{p-q};
\end{equation}
\item the forcing term $f$ in $L^r(0,T;L^m(\Omega))$ so that the couple $(m,r)$ verifies
\begin{equation}\label{F1}\tag{F1}
\frac{N\sigma}{m}+\frac{N(p-2)+p\sigma}{r}\le N(p-1)+ p\sigma.
\end{equation}
\end{itemize}
\begin{remark}
Note that if we consider $r=\infty$ in \eqref{F1}, then we have to ask that $m\ge \frac{N(q-(p-1))}{q}$ which is the needed condition to require on the source term of the stationary problem studied in \cite{GMP}.
\end{remark}
If, instead, we have
\[
\frac{2N}{N+1}<p<N\quad\text{and}\quad\max\left\{\frac{p}{2},\frac{p(N+1)-N}{N+2}\right\}<q<p-\frac{N}{N+1}
\] 
we assume that
\begin{itemize}
\item the initial datum $u_0$ is fixed in
 \begin{equation}\label{ID2}\tag{ID2}
u_0\in L^1(\Omega);
\end{equation}
\item the forcing term $f$ satisfies
\begin{equation}\label{F2}\tag{F2}
f\in L^1(Q_T).
\end{equation}
\end{itemize}
Note that the restriction $\frac{2N}{N+1}<p$ is needed in order to have $\frac{p}{2}<p-\frac{N}{N+1}$.\\
The borderline case in which $\ds \frac{2N}{N+1}<p<N$ and $\ds q=p-\frac{N}{N+1}$ will be briefly studied in commented in the Remark \ref{secq}, together with its own assumption.

\subsection{Statements of the main results and comments}

For the sake of clarity, we will collect on the real lines below the intervals of $q$ growth we deal with, emphasising also the assumptions on the data $u_0$ and $f$. 
\begin{center}
\begin{figure}[H]
\begin{tabular}{r|l}
\begin{tikzpicture}
\draw [very thick, dashed,color=red!70!black] (-1,0) to (2,0);
\draw [very thick, color=orange] (-1,-.2) to (2,-.2);
\end{tikzpicture}  
& $u_0\in L^\si(\Omega)$ and $f\in L^r(0,T;L^m(\Omega))$
\\
\begin{tikzpicture}
\draw [very thick, color=yellow] (-1,0) to (2,0);
\end{tikzpicture}
& $u_0\in L^1(\Omega)$ and $f\in L^1(Q_T)$ 
\end{tabular}
\captionsetup{labelformat=empty}
\caption{Colours legend}
\end{figure}
\end{center}
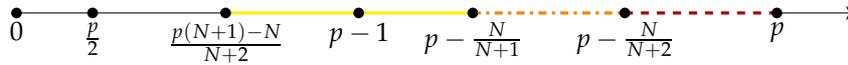
\begin{figure}[H]
\centering
\begin{tikzpicture}
\draw [thin] (0,0) to (2.75,0);
\draw [->,thin] (10,0) -- (11,0);
\draw [very thick,  dashed,color=red!70!black] (8,0) to (10,0);
\draw [very thick,dashdotted, color=orange] (6,0) to (8,0);
\draw [very thick, color=yellow] (2.75,0) to (6,0);
\fill (0,0) circle (2pt) node[below] 
{$0$};
\fill (1,0) circle (2pt) node[below] {$\frac{p}{2}$};
\fill (2.75,0) circle (2pt) node[below] {$\frac{p(N+1)-N}{N+2}$};
\fill (4.5,0) circle (2pt) node[below] {$p-1$};
\fill (6,0) circle (2pt) node[below] {$p-\frac{N}{N+1}$};
\fill (8,0) circle (2pt) node[below] {$p-\frac{N}{N+2}$};
\fill (10,0) circle (2pt) node[below] {$p$};
\end{tikzpicture}
\caption{The case $2\le p<N$}
\label{fig:1}
\end{figure}
Referring to the Figure \ref{fig:1} above, we point out that if $q\ge p-\frac{N}{N+2}$ then, according to \eqref{ID1}--\eqref{F1}, we have $\si\ge 2$, hence $u_0\in L^2(\Omega)$ and $f\in L^{p'}(0,T;W^{-1,p'}(\Omega))$; as $p-\frac{N}{N+1}<q<p-\frac{N}{N+2}$, then we have $1<\si<2$ and $f$ does not necessarily belong to the dual space. Finally, here the superlinearity threshold is given by $q=\frac{p(N+1)-N}{N+2}$ and, if $\frac{p(N+1)-N}{N+2}<q<p-\frac{N}{N+1}$, $L^1$ data are admitted.
\begin{figure}[H]
\centering
\begin{tikzpicture}
\draw [thin] (0,0) to (4.75,0);
\draw [->,thin] (10,0) -- (11,0);
\draw [very thick,  dashed,color=red!70!black] (8,0) to (10,0);
\draw [very thick, dashdotted,color=orange] (6,0) to (8,0);
\draw [very thick, color=yellow] (4.75,0) to (6,0);
\fill (0,0) circle (2pt) node[below] 
{$0$};
\fill (1.25,0) circle (2pt) node[below] 
{$p-1$};
\fill (2.85,0) circle (2pt) node[below] 
{$\frac{p(N+1)-N}{N+2}$};
\fill (4.75,0) circle (2pt) node[below] {$\frac{p}{2}$};
\fill (6,0) circle (2pt) node[below] {$p-\frac{N}{N+1}$};
\fill (8,0) circle (2pt) node[below] {$p-\frac{N}{N+2}$};
\fill (10,0) circle (2pt) node[below] {$p$};
\end{tikzpicture}
\caption{The case $\frac{2N}{N+1}<p< 2$}
\label{fig:2}
\end{figure}
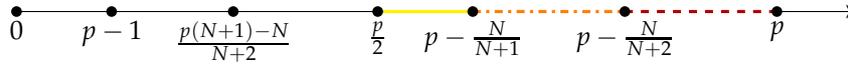
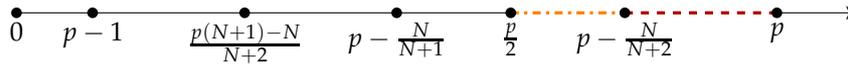
\begin{figure}[H]
\centering
\begin{tikzpicture}
\draw [thin] (0,0) to (6.5,0);
\draw [->,thin] (10,0) -- (11,0);
\draw [very thick,  dashed,color=red!70!black] (8,0) to (10,0);
\draw [very thick,dashdotted, color=orange] (6.5,0) to (8,0);
\fill (0,0) circle (2pt) node[below] 
{$0$};
\fill (1,0) circle (2pt) node[below] {$p-1$};
\fill (3,0) circle (2pt) node[below] {$\frac{p(N+1)-N}{N+2}$};
\fill (5,0) circle (2pt) node[below] {$p-\frac{N}{N+1}$};
\fill (6.5,0) circle (2pt) node[below] {$\frac{p}{2}$};
\fill (8,0) circle (2pt) node[below] {$p-\frac{N}{N+2}$};
\fill (10,0) circle (2pt) node[below] {$p$};
\end{tikzpicture}
\caption{The case $\frac{2N}{N+2}<p\le \frac{2N}{N+1}$}
\label{fig:3}
\end{figure}
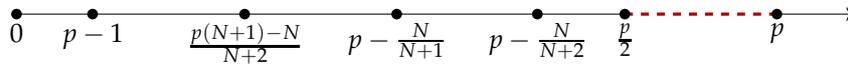
\begin{figure}[H]
\centering
\begin{tikzpicture}
\draw [thin] (0,0) to (8,0);
\draw [->,thin] (10,0) -- (11,0);
\draw [very thick,  dashed,color=red!70!black] (8,0) to (10,0);
\fill (0,0) circle (2pt) node[below] 
{$0$};
\fill (1,0) circle (2pt) node[below] {$p-1$};
\fill (3,0) circle (2pt) node[below] {$\frac{p(N+1)-N}{N+2}$};
\fill (5,0) circle (2pt) node[below] {$p-\frac{N}{N+1}$};
\fill (6.85,0) circle (2pt) node[below] {$p-\frac{N}{N+2}$};
\fill (8,0) circle (2pt) node[below] {$\frac{p}{2}$};
\fill (10,0) circle (2pt) node[below] {$p$};
\end{tikzpicture}
\caption{The case $1<p\le \frac{2N}{N+2}$}
\label{fig:4}
\end{figure}
Figures \ref{fig:2}, \ref{fig:3} and \ref{fig:4} show that, as $p$ becomes smaller than two, the superlinearity threshold changes into $q=\frac{p}{2}$. We underline that, depending on the ranges of $p$ above presented, such a value is smaller/greater with respect to the $L^1$ and $L^2$ thresholds of the initial data (namely, $q=p-\frac{N}{N+1}$ and $q=p-\frac{N}{N+2}$ respectively).\\

Roughly speaking, the figures above tell us that
\begin{enumerate}
\item[({\bf A})] if we make \emph{sharp assumptions} on the data, i.e. we assume \eqref{ID1} and \eqref{F1}, we expect to have at least
\begin{enumerate}
\item[({\bf A.1})] \emph{finite energy solutions} (see the red zone) if $1<p<N$ and either
\[
p-\frac{N}{N+2}\le q<p \,\,\text{ if }\,\, \frac{2N}{N+2}<p<N
\]
or
\[
\frac{p}{2}<q<p \,\,\text{ if }\,\, 1<p\le\frac{2N}{N+2}
\]
occurs;
\item[({\bf A.2})] \emph{infinite energy solutions} (see the orange zone) with $L^\si(\Omega)$ initial data, $\si>1$, if we assume
\[
\frac{2N}{N+2}<p<N\quad\text{and}\quad \max\left\{\frac{p}{2},p-\frac{N}{N+1} \right\}< q<p-\frac{N}{N+2}.
\]
\end{enumerate} 
\item[({\bf B})]  \emph{Infinite energy solutions} with $L^1$ data are admitted if
\[
\frac{2N}{N+1}<p<N \quad\text{and}\quad \max\left\{\frac{p}{2},\frac{p(N+1)-N}{N+2} \right\}< q<p-\frac{N}{N+1}.
\]
\end{enumerate}
As $q$ becomes too small, otherwise,  either
\[
p\ge 2\quad\text{and}\quad 0<q\le \frac{p(N+1)-N}{N+2}
\]
or
\[
1<p<2\quad\text{and}\quad 0<q\le \frac{p}{2}
\]
we fall within the sublinear growth case. This means that the problem behaves differently as we will show in Section\ref{secsub}.\\

Referring to the sketch given above, we here collect the statements of our main results.

\begin{theo}[Red zone]\label{1}
Let $\ds1<p<N$ and assume \eqref{A1}, \eqref{A2}, \eqref{A3}, \eqref{ID1}, \eqref{F1} and \eqref{H} with either
\[
p-\frac{N}{N+2}\le q<p \quad\text{if}\quad \frac{2N}{N+2}<p<N
\]
or
\[
\frac{p}{2}<q<p \quad\text{if}\quad 1<p\le\frac{2N}{N+2}.
\]
Then, there exists at least one finite energy solution of the problem \eqref{pb} (see Definition \ref{sol}). Moreover, this solution fulfils the following regularities:
\[
|u|^\beta\in L^p(0,T;W^{1,p}_0(\Omega))\quad\text{with}\quad \beta=\frac{\sigma-2+p}{p}
\] 
and
\[
u\in C([0,T];L^\sigma(\Omega)).
\]
\end{theo} 

\noindent
We point out that $\beta\ge 1$ if $q\ge p-\frac{N}{N+2}$ and $p>\frac{2N}{N+2}$, $\beta>1$ if $q>\frac{p}{2}$ and
$p\le\frac{2N}{N+2}$.

\begin{theo}[Orange zone]\label{2}
Let $\ds\frac{2N}{N+2}<p<N$ and assume \eqref{A1}, \eqref{A2}, \eqref{A3}, \eqref{ID1}, \eqref{F1} and \eqref{H} with
\[
\max\left\{ \frac{p}{2},p-\frac{N}{N+1}\right\}<  q<p-\frac{N}{N+2}.
\]
Then, there exists at least one solution $u$ of the problem \eqref{pb} (see Definition \ref{defrin}). Moreover, such a solution fulfils the following regularities:
\begin{equation*}
(1+|u|)^{\beta-1}u\in L^p(0,T;W_0^{1,p}(\Omega))\quad\text{with}\quad \beta=\frac{\sigma-2+p}{p}
\end{equation*}
and
\[
u\in C([0,T];L^{\sigma}(\Omega)).
\]
\end{theo}

\noindent
Note that having $p-\frac{N}{N+1}<q<p-\frac{N}{N+2}$ implies that $1<\si<2$. The restriction $p>\frac{2N}{N+2}$ is necessary in order to have $\frac{p}{2}<p-\frac{N}{N+2}$.

\begin{theo}[Yellow zone]\label{3}
Let $\ds\frac{2N}{N+1}<p<N$ and assume \eqref{A1}, \eqref{A2}, \eqref{A3}, \eqref{ID2}, \eqref{F2} and  \eqref{H} with
\[
\frac{p(N+1)-N}{N+2}<  q< p-\frac{N}{N+1}.
\]
Then, there exists at least one renormalized solution $u$ of the problem \eqref{pb} as in Definition \ref{defrin2}. Moreover, the following regularity holds:
\[
u\in C([0,T];L^1(\Omega)).
\]
\end{theo}
\noindent
The result in Theorem \ref{3} could as well be referred to case $p\le \frac{2N}{N+1}$. However, this case is completely contained in Theorem \ref{Tsub} since it corresponds to a sublinear growth.

\subsection{Plan of the paper}

We discuss the finite energy case in Section \ref{secplap}, i.e. we require that the gradient growth rate and the data are as in ({\bf{A.1}}).
Note that the ranges of $q$ and the definition of $\si$ ensure that $\sigma\ge 2$. This Section contains the proof of Theorem \ref{1}.

Section \ref{secrin} is devoted to the study of the growth interval in {({\bf {A.2}})}.
Since this range of $q$ implies that $1< \sigma<2$, solutions will have not finite energy. The proof of Theorem \ref{2} is here presented.\\
The particular case $\ds q=p-\frac{N}{N+1}$ is briefly outlined at the end of this Section, see Remark \ref{secq}.

We discuss the last superlinear interval {({\bf B})}
in Section \ref{secrin2} where renormalized solutions are considered. Note that this range of $q$ implies that $\sigma$, $m$ and $r$ would become strictly less than $1$. This means that measure data can be considered. However, we will take into account only $L^1$ data. Theorem \ref{3} is here proved.


We dedicate our last Section \ref{secsub} to the study of an existence result (see Theorem \ref{Tsub}) in the case of small values of $p$, namely $\ds 1<p<2$, and when the r.h.s. exhibits a  sublinear growth, i.e. we assume that  
\[
0<q\le\frac{p}{2}.
\]
In this way, we fill a gap with the results existing in the literature \cite{Po,DNFG} devoted to the sublinear or linear growth of $H(t,x,\xi)$.

Finally, we conclude  collecting in the conclusive Appendices some needed tools and useful results. More precisely, Appendix \ref{secpbn} contains the definition of the approximating problem we will consider during the paper and some preliminary results. Lemmas concerning Marcinkievicz estimates are contained in Appendix \ref{marc}. Appendix \ref{sharp} is devoted to the proof of a nonexistence result when initial data $u_0\in L^\nu(\Omega)$ for $\nu<\si$ and $f=0$ are considered.

\subsection*{Notation}
We will represent the constant due to the Sobolev's embedding by $c_S$ while $c$ will stand for a positive constant which may vary line to line during the proofs and is independent of the parameter $n$ used for the approximating problem.

We will need some auxiliary functions which are in the following defined:
\[
G_k(v)=(|v|-k)_+\text{sign}(v),\qquad T_k(v)=v-G_k(v)=\max\{-k,\min\{k,v\}\}.
\]
\begin{figure}[H]
\centering
\begin{tikzpicture}
\draw[->] (0,0) -- (3,0) node[anchor=north west] {v};
\draw[->] (0,0) -- (0,2) node[left] {{\color{red!70!black}$T_k(v)$}, {\color{orange}$G_k(v)$}};
\draw[<-] (-3,0) -- (0,0);
\draw[<-] (0,-2) -- (0,0);
\draw[very thick, red!70!black] (-1,-1) -- (1,1);
\draw[very thick, red!70!black] (1,1) -- (2,1);
\draw[very thick, red!70!black] (-1,-1) -- (-2,-1);
\draw[very thick, dashed, red!70!black] (2,1) -- (3,1);
\draw[very thick, dashed, red!70!black] (-2,-1) -- (-3,-1);
\draw[very thick, orange] (-1,0) -- (1,0);
\draw[very thick, orange] (1,0) -- (2,1);
\draw[very thick, orange] (-1,0) -- (-2,-1);
\draw[very thick, dashed, orange] (2,1) -- (3,2);
\draw[very thick, dashed, orange] (-2,-1) -- (-3,-2);
\draw[dashed] (1,1) -- (1,0);
\draw[dashed] (-1,-1) -- (-1,0);
\draw[dashed] (2,1) -- (2,0);
\draw[dashed] (-2,-1) -- (-2,0);
\draw[dashed] (1,1) -- (0,1);
\draw[dashed] (-1,-1) -- (0,-1);
\fill (1,0) circle (1.5pt) node[below] {$k$};
\fill (2,0) circle (1.5pt) node[below] {$k+1$};
\fill (-1,0) circle (1.5pt) node[above] {$-k$};
\fill (-2,0) circle (1.5pt) node[above] {$-(k+1)$};
\fill (0,1) circle (1.5pt) node[left] {$k$};
\fill (0,-1) circle (1.5pt) node[right] {$-k$};
\end{tikzpicture}
\caption{The functions $G_k(v)$ and $T_k(v)$}\label{fig:GkTk}
\end{figure}
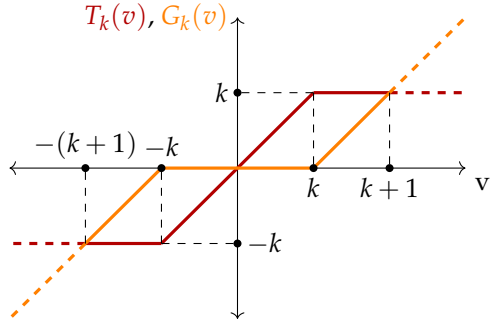
We denote the sets where $G_k(v(t))$ and $G_k(v)$ are different from zero by
\[
A_k^{t}:=\{x\in \Omega: \, |v(t,x)|>k \} \,\, \text{for fixed } t\in[0,T]\quad \text{and} \quad A_k:=\{(t,x)\in Q_T: \, |v(t,x)|>k\}.
\]

\section{Solutions of finite energy}\label{secplap}
\setcounter{equation}{0}

We begin this Section presenting the definition of {finite energy solution}.
\begin{definition}\label{sol}
A finite energy solution $u$ of \eqref{P} is a real valued function $u$ belonging to
\[
u\in L^p(0,T;W^{1,p}_0(\Omega)), \quad1<p<N,
\]
which satisfies the weak formulation:
\begin{equation*}
-\integrale u_0(x)\vp(0,x)\,dx+\iint_{Q_T}-u\vp_t+a(t,x,u,\N u)\cdot\N\vp\,dx\,dt=\iint_{Q_T}H(t,x,\N u)\vp\,dx\,dt
\end{equation*}
for every test function $\vp\in L^p(0,T;W_0^{1,p}(\Omega))\cap L^{\infty}(Q_T)$ such that $\vp_t\in L^{p'}(Q_T)$ and $\vp(T,x)=0$.
\end{definition}
We prove the existence of such a solution proceeding by approximation and thus we deal with a sequence of solutions of \eqref{Pn}, see Appendix \ref{secpbn}.

\subsection{The a priori estimate}
\begin{theorem}\label{sapp}
Assume $1<p<N$, \eqref{A1}, \eqref{A2}, \eqref{H} with $p-\frac{N}{N+2}\le q<p$ if $\frac{2N}{N+2}<p<N$ and $\frac{p}{2}<q<p$ if $1<p\le \frac{2N}{N+2}$, \eqref{F1}, \eqref{ID1} and let $\{u_n\}_n$ be a sequence of solutions of \eqref{Pn}. Then 
\begin{itemize}
\item $\{u_n\}_n$ is uniformly bounded in $L^\infty(0,T;L^{\sigma}(\Omega))\cap L^p(0,T;W^{1,p}_0(\Omega))$;
\item  $\{|u_n|^\beta\}_n$, $\beta=\frac{\sigma+p-2}{p}$, is uniformly bounded in $L^p(0,T;W^{1,p}_0(\Omega))$.
\end{itemize}
Moreover, the following inequality holds:
\begin{equation}\label{disapp}
\sup_{t\in [0,T]}\|u_n(t)\|_{L^{\sigma}(\Omega)}^{\sigma}+\|\nabla((1+|u_n|)^{\beta})\|_{L^p(Q_T)}^p\le M
\end{equation}
where the constant $M$ depends on $\alpha$, $p$, $q$, $\gamma$, $N$, $|\Omega|$, $T$, $u_0$, $f$ and remains bounded when $u_0$ and $f$ vary in sets which are bounded and equi-integrable, respectively, in $L^{\sigma}(\Omega)$ and $L^r(0,T;L^m(\Omega))$.
\end{theorem}
\proof
{\textit{Part $1$.}}\\
Let us consider the change of variable $\displaystyle w_n=e^{-t}u_n$ so that the problem \eqref{Pn} becomes
\begin{equation}\label{wn}
(w_n)_t+w_n-\dive \tilde{ a}(t,x,w_n,\nabla w_n)=\tilde {H}_n(t,x,\nabla w_n)
\end{equation}
where $\tilde{a}(t,x,u,\xi)=e^{-t}{a}(t,x,e^tu,e^t\xi)$ and $\tilde{H}_n(t,x,\xi)=e^{-t}H_n(t,x,e^t\xi)$. Note that \eqref{A1}-\eqref{A2} and \eqref{H} still hold with different constants (all depending on $T<\infty$), say $\tilde\al$, $\tilde \lambda$ and $\tilde{\gamma}$ respectively. We underline that $w_n(0,x)=u_{0,n}(x)$ for every $n\in \mathbb{N}$ and homogeneous Dirichlet boundary conditions are still satisfied in the same space (i.e. $w_n=0$ over $(0,T)\times\partial \Omega$). Moreover, we observe that $w_n$ and $u_n$ have the same behaviour for finite time: indeed, $u_n\le e^T w_n$ pointwise. This fact allows us to say that the bounds satisfied by $\{w_n\}_n$ hold (for finite time) for $\{u_n\}_n$ as well.
We point out that the change of variable makes a zero order term appear and this term helps us dealing with $f$.

We multiply the equation in \eqref{wn} by $|G_k(w_n)|^{\sigma-2}G_k(w_n)$ and integrate over $Q_t$. Thus, thanks to the assumptions \eqref{A1} and \eqref{H}, we have:
\begin{equation*}
\begin{array}{c}
\ds
\frac{1}{\sigma}\int_{\Omega}|G_k(w_n(t))|^{\sigma}\,dx+
k\iint_{Q_t} |G_k(w_n)|^{\sigma-1}\,dx\,ds
+\tilde\alpha(\sigma-1)
\iint_{Q_t} |\nabla G_k(w_n) |^p |G_k(w_n)|^{\sigma-2}\,dx\,ds\\
[3mm]\ds
\le\tilde\gamma  \iint_{Q_t} |\nabla G_k(w_n)|^q |G_k(w_n)|^{\sigma-1}\,dx\,ds
+ \iint_{Q_t}|f|\chi_{\{|f|>k\}}|G_k(w_n)|^{\sigma-1}\,dx\,ds\\
[3mm]\ds
+\iint_{Q_t}|f|\chi_{\{|f|\le k\}}|G_k(w_n)|^{\sigma-1}\,dx\,ds+\frac{1}{\sigma}\int_{\Omega}|G_k(u_0)|^{\sigma}\,dx.
\end{array}
\end{equation*}
The change of variable allows us to simplify as below:
\begin{equation*}
\begin{array}{c}
\ds
\frac{1}{\sigma}\int_{\Omega}|G_k(w_n(t))|^{\sigma}\,dx
+\tilde\alpha\frac{\sigma-1}{\beta^p}
\iint_{Q_t} |\nabla [|G_k(w_n)|^\beta] |^p\,dx\,ds
\\[3mm]
\ds
\le\tilde\gamma \iint_{Q_t} |\nabla G_k(w_n)|^q |G_k(w_n)|^{\sigma-1}\,dx\,ds
+ \iint_{Q_t}|f|\chi_{\{|f|>k\}}|G_k(w_n)|^{\sigma-1}\,dx\,ds+\frac{1}{\sigma}\int_{\Omega}|G_k(u_0)|^{\sigma}\,dx.
\end{array}
\end{equation*}
Estimating the first integral in the r.h.s. using H\"older's inequality with indices $(\frac{p}{q}, \frac{p}{p-q})$ we get
\begin{equation}\label{3Hpp}
\begin{array}{c}
\ds
\iint_{Q_t} |\nabla G_k(w_n)|^q |G_k(w_n)|^{\sigma-1}\,dx\,ds \\
[3mm]
\ds
\le  \tilde\gamma\iint_{Q_t} |\nabla G_k(w_n)|^q |G_k(w_n)|^{q(\beta-1)} |G_k(w_n)|^{(\sigma-1)\frac{p-q}{p}+\frac{q}{p}}\,dx\,ds\\
[3mm]
\ds
\le \frac{1}{\beta^q} \int_0^t \biggl[\biggl(\int_{\Omega} |\nabla [|G_k(w_n)|^{\beta}]|^p \,dx\biggr)^{\frac{q}{p}}
\cdot 
\biggl(\int_{\Omega}  |G_k(w_n)|^{\sigma-1+\frac{q}{p-q}}\,dx\biggr)^{\frac{p-q}{p}}\biggr]\,ds.
\end{array}
\end{equation}
Moreover, being $\sigma-1+\frac{q}{p-q}=p\beta+\frac{p\sigma}{N}$ by definitions of $\beta$ and $\sigma$, we can apply again H\"older's inequality with $\left(\frac{p^*}{p},\frac{N}{p}\right)$ and Sobolev's embedding too, so that we obtain
\begin{equation*}
\begin{array}{c}
\ds
\tilde\gamma  \iint_{Q_t} |\nabla G_k(w_n)|^q |G_k(w_n)|^{\sigma-1}\,dx\,ds\notag\\
[3mm]
\ds
\le \frac{\tilde\gamma}{\beta^q}\int_0^t\biggl(\int_{\Omega} |\nabla [|G_k(w_n)|^{\beta}]|^p \,dx\biggr)^{\frac{q}{p}}\biggl(\int_{\Omega}  |G_k(w_n)|^{p^*\beta}\,dx\biggr)^{\frac{p-q}{p^*}}
 \biggl(\int_{\Omega}  |G_k(w_n)|^{\sigma}\,dx\biggr)^{\frac{p-q}{N}}\,ds\\
 [3mm]
 \ds
 \le\frac{c_S\tilde\gamma}{\beta^q}\int_0^t\biggl(\int_{\Omega} |\nabla [|G_k(w_n)|^{\beta}]|^p \,dx\biggr)\biggl(\int_{\Omega}  |G_k(w_n)|^{\sigma}\,dx\biggr)^{\frac{p-q}{N}}\,ds.
\end{array}
\end{equation*}

As far as the integral involving $f$ is concerned, twice applications of H\"older's inequalities with indices $(m,m')$ and $(r,r')$ give us
\begin{align*}
\iint_{Q_t}|f|\chi_{\{|f|>k\}}|G_k(w_n)|^{\sigma-1}\,dx\,ds&\le \||f|\chi_{\{|f|>k\}}\|_{L^r(0,T;L^m(\Omega))}
\|G_k(w_n)\|_{L^{r'(\si-1)}(0,t;L^{m'(\si-1)}(\Omega))}^{\sigma-1}\\
&=\||f|\chi_{\{|f|>k\}}\|_{L^r(0,T;L^m(\Omega))}
\||G_k(w_n)|^{\beta}\|_{L^{r'\frac{\si-1}{\beta}}(0,t;L^{m'\frac{\si-1}{\beta}}(\Omega))}^{\frac{\si-1}{\beta}}.
\end{align*}
We go further invoking Theorem \ref{teoGN} with $v=|G_k(w_n)|^{\beta}$, $h=\frac{\si}{\beta}$ and $\eta=p$. We notice again that $h<p^*$ since $q>\frac{p}{2}$. Then, we have that
\[
|G_k(w_n)|^{\beta}\in L^y(0,T;L^w(\Omega))\quad\forall n\in \mathbb{N}
\]
where the couple $(w,y)$ satisfies the relation in \eqref{rel}.
In particular, the inequality in \eqref{GN} becomes
\[
\int_0^t \||G_k(w_n(s))|^{\beta}\|_{L^{w}(\Omega)}^y\,ds\le c(N,p,w)\||G_k(w_n)|^{\beta}\|_{L^{\infty}(0,t;L^{\frac{\si}{\beta}}(\Omega))}^{y-p}\int_0^t \|\N[|G_k(w_n(s))|^{\beta}]\|_{L^p(\Omega)}^p\,ds.
\]
Algebraic computations show that the hypotheses \eqref{F1} ensures that
\[
w\ge m'\frac{\si-1}{\beta}\quad\text{and}\quad y\ge r'\frac{\si-1}{\beta}
\]
and thus we can proceed estimating as below:
\begin{equation*}
\begin{array}{c}
\ds
\||f|\chi_{\{|f|>k\}}\|_{L^r(0,T;L^m(\Omega))}
\||G_k(w_n)|^{\beta}\|_{L^{y}(0,T;L^{w}(\Omega))}^{\frac{\si-1}{\beta}}\\
[3mm]\ds
\le c\||f|\chi_{\{|f|>k\}}\|_{L^r(0,T;L^m(\Omega))}\left(
\|G_k(w_n)\|_{L^{\infty}(0,t;L^{\si}(\Omega))}^{\beta(y-p)}\int_0^t \|\N[|G_k(w_n(s))|^{\beta}]\|_{L^p(\Omega)}^p\,ds
\right)^{\frac{\si-1}{y\beta}}\\
[3mm]\ds
\le c_1\|G_k(w_n)\|_{L^{\infty}(0,t;L^{\si}(\Omega))}^{\beta(y-p)}\int_0^t \|\N[|G_k(w_n(s))|^{\beta}]\|_{L^p(\Omega)}^p\,ds+c_2\||f|\chi_{\{|f|>k\}}\|_{L^r(0,T;L^m(\Omega))}^{\frac{y\beta}{y\beta-(\si-1)}}
\end{array}
\end{equation*}
where the last inequality is due to Young's one with indices $\left( \frac{y\beta}{\si-1},\frac{y\beta}{y\beta-(\si-1)}\right)$. 

We collect our previous estimates in the following inequality:
\begin{equation}\label{disfin}
\begin{array}{c}
\ds
\frac{1}{\sigma}\int_{\Omega}|G_k(w_n(t))|^{\sigma}\,dx+
\tilde\alpha\frac{\sigma-1}{\beta^p}
\iint_{Q_t} |\nabla [|G_k(w_n(s))|^\beta] |^p \,dx\,ds\\
[3mm]\ds
\le 
 \left[c_1\|G_k(w_n)\|_{L^{\infty}(0,t;L^{\si}(\Omega))}^{\beta(y-p)}+\frac{\tilde\gamma  c_S}{\beta^q}\|G_k(w_n)\|_{L^{\infty}(0,t;L^{\si}(\Omega))}^{\si\frac{p-q}{N}}\right]\int_0^t
\|\N [|G_k(w_n(s))|^{\beta}]\|_{L^{p}(\Omega)}^{p}\,ds
\\
[3mm]\ds
+c_2\||f|\chi_{\{|f|>k\}}\|_{L^r(0,T;L^m(\Omega))}^{\frac{y\beta}{y\beta-(\si-1)}}+\frac{1}{\sigma}\int_{\Omega}|G_k(u_0)|^{\sigma}\,dx.
\end{array}
\end{equation}
\noindent
The next steps are aimed at absorbing the gradient term in the r.h.s. to the l.h.s..\\
We fix a value $\delta_0$ such that $\ds 2\max\left\{
\frac{\tilde\gamma c_S}{\beta^q} {\delta_0}^{\frac{p-q}{N}},
c_1\delta_0^{\frac{\beta(y-p)}{\si}}\right\}
=\tilde\alpha\frac{\sigma-1}{2\beta^p}$ and a value $ k_0$ large enough so that
\begin{equation}\label{u0}
\|G_k(u_0)\|_{L^{\sigma}(\Omega)}^{\sigma}<\frac{\delta_0}{2}\qquad\forall k\ge k_0
\end{equation} 
and
\begin{equation}\label{f0}
\sigma c_2\||f|\chi_{\{|f|>k\}}\|_{L^r(0,T;L^m(\Omega))}^{\frac{y\beta}{y\beta-(\si-1)}}< \frac{\delta_0}{2}\qquad\forall k\ge k_0.
\end{equation}
Moreover, for $k\ge k_0$, we set
\begin{equation*}
T^*:=\sup\{s\in [0,T]: \,\|G_k(w_n(t))\|_{L^{\sigma}(\Omega)}^{\sigma}\le \delta_0 \,\,\forall   t\le s  \}.
\end{equation*}
Since, thanks to \cite[Theorem $1.1$]{G}, $\{w_n\}\subseteq C([0,T];L^{\nu}(\Omega))$ for every $1\le\nu<\infty$, we have that $T^*>0$ due to \eqref{u0}.

If we suppose that $t\le T^*$ in \eqref{disfin}, then the definition of $\delta_0$ implies that
\begin{equation}\label{datotp}
\begin{array}{c}
\ds
\frac{1}{\sigma}\int_{\Omega}|G_k(w_n(t))|^{\sigma}\,dx+\tilde\alpha \frac{\sigma-1}{2\beta^p}
\int_0^t\int_{\Omega}  |\nabla [|G_k(w_n)|^{\beta}] |^p  \,dx\,ds\\
[3mm]\ds\le\frac{1}{\sigma}\int_{\Omega}|G_k(u_0)|^{\sigma}\,dx
+c_2\||f|\chi_{\{|f|>k\}}\|_{L^r(0,T;L^m(\Omega))}^{\frac{y\beta}{y\beta-(\si-1)}}
\end{array}
\end{equation}
for every $k\ge k_0$. We can extend the inequality \eqref{datotp} to the whole interval $[0,T]$ observing that, if $t=T^*<T$, then \eqref{u0} and \eqref{f0} lead to
\[
\int_{\Omega}|G_k(w_n(T^*))|^{\sigma}\,dx<\delta_0
\]
which is in contrast with the definition of $T^*$ because of the continuity regularity $C([0,T];L^{\sigma}(\Omega))$. Therefore, we have that $T^*=T$ and \eqref{datotp} holds for all $t\le T$, that is
\begin{equation}\label{thGkpp}
\sup_{t\in (0,T)}\|G_k(u_n(t))\|_{L^{\sigma}(\Omega)}^{\sigma}+\|\nabla[|G_k(u_n)|^{\beta}]\|_{L^p(Q_{T})}^p\le M_1 \quad\forall k\ge k_0.
\end{equation} 

The proof of the Theorem will be concluded once we show that $|\N [|T_{k_0}(w_n)|^\beta]|$ satisfies a bound like the one proved in \eqref{thGkpp}.
With this purpose, we multiply the equation in \eqref{wn} for $|T_{k_0}(w_n)|^{\sigma-2}T_{k_0}(w_n)$ and integrate over $Q_t$, so that we have:
\begin{equation*}
\begin{array}{c}
\ds
\int_{\Omega}\Theta_{k_0}(w_n(t))\,dx+\iint_{Q_t}|w_n||T_{k_0}(w_n)|^{\sigma-1}\,dx\,ds
+\tilde\alpha\frac{ (\sigma-1)}{\beta^p}\iint_{Q_t}|\nabla ( |T_{k_0}(w_n)|^{\beta})|^p\,dx\,ds\\
[3mm]\ds
\le \tilde\gamma \iint_{Q_t}|\nabla w_n|^q|T_{k_0}(w_n)|^{\sigma-1}\,dx\,ds
+\iint_{Q_t}|f||T_{k_0}(w_n)|^{\sigma-1}\,dx\,ds+\int_{\Omega}\Theta_{k_0}(u_0)\,dx
\end{array}
\end{equation*}
where 
\[
\Theta_{k_0}(s)=\int_0^s |T_{k_0}(z)|^{\sigma-2}T_{\bar{k}}(z)\,dz.
\] 
The last integral in the r.h.s. is uniformly bounded in $n$ thanks to the assumption on the initial datum $u_0$. 
As far as the first integral is concerned, we use the decomposition $w_n=G_{k_0}(w_n)+T_{k_0}(w_n)$ and estimate as below
 \begin{align*}
 \tilde\gamma \iint_{Q_t}|\nabla w_n|^q|T_{k_0}(w_n)|^{\sigma-1}\,dx\,ds&
\le \frac{\tilde\gamma}{\beta^q}\iint_{Q_t}|\nabla [|w_n|^\beta]|^q|T_{k_0}(w_n)|^{(\sigma-1)\frac{p-q}{p}+\frac{q}{p}}\,dx\,ds\\
&\le  c\biggl[\iint_{Q_t}|\nabla [|G_{k_0}(w_n)|^\beta]|^q\,dx\,ds+\iint_{Q_t}|\nabla [|T_{k_0}(w_n)|^\beta]|^q\,dx\,ds\biggr]
 \end{align*}
where $c=c(\si,p,q,k_0)$. Twice applications of Young's inequality with $\left(\frac{p}{q},\frac{p}{p-q}\right)$ and the bound obtained in \eqref{thGkpp} give us
\begin{equation*}
\iint_{Q_t}|\nabla [|G_{k_0}(w_n)|^\beta]|^q\,dx\,ds
\le c\iint_{Q_T}|\nabla [|G_{k_0}(w_n)|^\beta]|^p\,dx\,ds+c
\le c_0[M_1+1]
\end{equation*}
and
\begin{equation*}
\iint_{Q_t}|\nabla [|T_{k_0}(w_n)|^\beta]|^q\,dx\,ds\le \tilde{\alpha}\frac{\sigma-1}{2\beta^p}\iint_{Q_t}|\nabla [|T_{k_0}(w_n)|^\beta]|^p\,dx\,ds+\tilde{c}_0
\end{equation*} 
where both $c_0$ and $\tilde{c}_0$ depend on $|\Omega|$, $T$ and $k_0$.\\
Finally, since
\[
\iint_{Q_t}|f||T_{k_0}(w_n)|^{\sigma-1}\,dx\,ds\le \bar{c}_0\|f\|_{L^r(0,T;L^m(\Omega))}
\]
with $\bar{c}_0=\bar{c}_0(T,|\Omega|,k_0)$, we collect all the previous estimates in the following inequality:
\begin{equation*}
\begin{array}{c}
\ds
\int_{\Omega}\Theta_{k_0}(w_n(t))\,dx+\tilde\alpha\frac{ \sigma-1}{2\beta^p}\iint_{Q_t}|\nabla ( |T_{k_0}(w_n)|^{\beta})|^p\,dx\,ds\\
[3mm]\ds
\le c\left[M_1+1 +\|f\|_{L^r(0,T;L^m(\Omega))}\right]
 +\int_{\Omega}|u_0|^{\sigma}\,dx.
\end{array}
\end{equation*}
In the end, we have found that:
\[
\|\nabla[|T_{k_0}(w_n)|^{\beta}]\|_{L^p(Q_T)}^p\le M_2
\]
where $M_2=M_2(\delta_0,k_0,T,|\Omega|,f,u_0)$ besides the parameters given by the problem.\\
Then, the inequality \eqref{disapp} follows with $M$ depending on $\alpha$, $p$, $q$, $\gamma$, $N$, $|\Omega|$, $T$ and $k_0$. In particular, since $k_0=k_0(\delta_0)$, $M$ remains bounded when $u_0$ and $f$ vary in sets which are bounded and equi-integrable, respectively, in $L^{\sigma}(\Omega)$ and $L^r(0,T;L^m(\Omega))$.\\

\noindent
{\textit{Part $2$.}}\\
We observe that the boundedness of $\{|u_n|^{\beta}\}_n$ in $L^p(0,T;W^{1,p}_0(\Omega))$ does not provide that $\{u_n\}_n$ is uniformly bounded in $L^p(0,T;W^{1,p}_0(\Omega))$ as well. However, choosing $w_n$ as test function and thanks to Young's inequality with $\left(\frac{p}{q},\frac{p}{p-q}\right)$, we get
\begin{equation*}
\begin{array}{c}
\ds
\frac{1}{2}\int_{\Omega}|w_n(t)|^2\,dx+\iint_{Q_t}|w_n|^2\,dx\,dt+\tilde{\al}\iint_{Q_t}|\N w_n|^p\,dx\,dt\\
[3mm]\ds
\le \tilde{\gamma}\iint_{Q_t}|\N w_n|^q|w_n|\,dx\,dt+\iint_{Q_t}|f||w_n|\,dx\,dt+\frac{1}{2}\int_{\Omega}|u_0|\,dx\\
[3mm]\ds
\le \frac{\tilde{\al}}{2}\iint_{Q_t}|\N w_n|^p\,dx\,dt+c\iint_{Q_t}|w_n|^{\frac{p}{p-q}}\,dx\,dt+\iint_{Q_t}|f||w_n|\,dx\,dt+\frac{1}{2}\int_{\Omega}|u_0|^2\,dx
\end{array}
\end{equation*}
from which
\[
\begin{array}{c}
\ds
\frac{1}{2}\int_{\Omega}|w_n(t)|^2\,dx+\iint_{Q_t}|w_n|^2\,dx\,dt+\frac{\tilde{\al}}{2}\iint_{Q_t}|\N w_n|^p\,dx\,dt\\
[3mm]\ds
\le c\iint_{Q_t}|w_n|^{\frac{p}{p-q}}\,dx\,dt+\iint_{Q_t}|f||w_n|\,dx\,dt+\frac{1}{2}\int_{\Omega}|u_0|^2\,dx.
\end{array}
\]
Then, since $\{w_n\}_n$ is bounded in $L^{\frac{p}{p-q}}(Q_T)$ (indeed, $\frac{p}{\beta(p-q)}\le p\frac{N+\frac{\si}{\beta}}{N}$ being $\si\ge 2$) and in $L^{r'}(0,T;L^{m'}(\Omega))$ (since $\frac{m'}{\beta}\le \frac{m'}{\beta}(\si-1)$ and similar for $r'$), the assertion follows.
\endproof
\begin{corollary}\label{simon}
Assume \eqref{A1}, \eqref{A2}, \eqref{H} with $p-\frac{N}{N+2}\le q<p$ if $\frac{2N}{N+2}<p<N$ and $\frac{p}{2}<q<p$ if $1<p\le \frac{2N}{N+2}$, \eqref{F1} and \eqref{ID1}. Moreover, let $\{u_n\}_n$ be a sequence of solutions of \eqref{Pn}. Then, up to subsequences, $u_n$ converges strongly to some function $u$ in $L^p(Q_T)$.
\end{corollary}
\proof
Standard compactness results (see \cite[Corollary $4$]{S}) guarantee that, up to subsequences, $u_n$ converges strongly to $u$ in $L^p(Q_T)$. We here recall the hypotheses of \cite[Corollary $4$]{S}: 
\begin{center}
Let $X$, $B$ and $Y$ be Banach spaces such that
\[
X\hookrightarrow B \hookrightarrow Y
\]
where the embedding $X\hookrightarrow B$ is compact. Then, if $\{ u_n\}_n\subseteq L^p(0,T;X)$, $1\le p <\infty$, and $\{( u_n)_t\}_n\subseteq L^1(0,T;Y)$, we have that $\{ u_n\}_n$ is relatively compact in $L^p(0,T;B)$.
\end{center}
We thus apply the result above for $p>1$, $X=W^{1,p}_0(\Omega)$, $B=L^p(\Omega)$ and $Y=W^{-1,s'}(\Omega)$ and $s$ greater than $N$.
\endproof

\subsection{The a.e. convergence of the gradient}
We prove here that the a.e. convergence $\N u_n\to \N u $ holds. This last step is essential in order to prove the desired existence result: indeed, even if we would deal with linear operator in the l.h.s., the nature of the r.h.s. requires this step.

\begin{proposition}\label{aeD}
Assume $1<p<N$, \eqref{A1}, \eqref{A2}, \eqref{A3}, \eqref{H} with $p-\frac{N}{N+2}\le q<p$ if $\frac{2N}{N+2}<p<N$ and $\frac{p}{2}<q<p$ if $1<p\le \frac{2N}{N+2}$, \eqref{F1} and \eqref{ID1}. Then there exists a subsequence (still denoted by $u_n$) and a function $u$ such that
\[
\nabla u_n \to \nabla u\quad\text{a.e. }\,\,Q_T.
\]  
Moreover $H_n(t,x,\N u_n)$ converges strongly to $H(t,x,\N u)$ in $L^1(Q_T)$.
\end{proposition}
\proof 
Theorem \ref{sapp} ensures that $\{|\N u_n|\}_n$ is bounded in $ L^p(Q_T)$. In particular, this means that the r.h.s. is bounded in $L^1(Q_T)$. Then, we can reason as in \cite[Theorem $3.3$]{BDGO} and deduce the a.e. convergence of the gradient.

Now, we want to apply the Vitali Theorem in order to get the strong convergence of $H_n(t,x,\N u_n)$. The a.e. convergence of $H_n(t,x,\N u_n)$ to $H(t,x,\N u)$ holds by the a.e. convergence of the gradient seen above. It remains only to show that
\[
\lim_{|E|\to 0}\sup_n\iint_E |H_n(t,x,\N u_n)|\,dx\,dt\to 0,
\]
$E\subset Q_T$. The assumption \eqref{H} ensures that
\begin{align*}
\iint_E |H_n(t,x,\N u_n)|\,dx\,dt&\le \iint_E |\N u_n|^q\,dx\,dt+\iint_E |f|\,dx\,dt.
\end{align*}
and thus, having $\{|\N u_n|^p\}_n$ uniformly bounded in $L^{1}(Q_T)$, $q<p$ and \eqref{F1}, the assertion follows.
\endproof

\subsection{The existence result} 
We are now able to prove the following existence result.
\begin{theorem}\label{Tfe}
Assume $1<p<N$, \eqref{A1}, \eqref{A2}, \eqref{A3}, \eqref{H} with $p-\frac{N}{N+2}\le q<p$ if $\frac{2N}{N+2}<p<N$ and $\frac{p}{2}<q<p$ if $1<p\le \frac{2N}{N+2}$, \eqref{F1} and \eqref{ID1}. Then, there exists at least one finite energy solution $u\in L^p(0,T;W^{1,p}_0(\Omega))$  of \eqref{P} in the sense of Definition \ref{sol} which satisfies
\begin{equation}\label{ctfe}
u\in C([0,T];L^\sigma(\Omega))
\end{equation}
and
\begin{equation}\label{potfe}
|u|^\beta\in L^p(0,T;W^{1,p}_0(\Omega))\quad\text{with}\quad \beta=\frac{\sigma-2+p}{p}.
\end{equation}
\end{theorem} 
\proof
Let $\{ u_n\}_n\subseteq L^p(0,T;W^{1,p}_0(\Omega))$ be the sequence of solutions of \eqref{Pn}.\\
Theorem \ref{sapp} implies that $\{ u_n\}_n$ and $\{|u_n|^\beta\}_n$ are uniformly bounded in $L^\infty(0,T;L^{\sigma}(\Omega))\cap L^p(0,T;W^{1,p}_0(\Omega))$ and in $L^p(0,T;W^{1,p}_0(\Omega))$ respectively and the inequality \eqref{disapp} holds. Moreover, Corollary \ref{simon} ensures that $u_n\to u$ in $L^p(Q_T)$ (up to subsequences) and, in particular, $u_n\to u$ a.e. (again, up to subsequences). Then, since $\nabla u_n\to \nabla u$ a.e., we let $n\to \infty$ in \eqref{disapp} and conclude that
\[
\sup_{t\in (0,T)}\|u(t)\|_{L^\sigma(\Omega)}^\sigma+\int_0^T\|\N[(1+|u|)^\beta]\|_{L^p(\Omega)}^p\le M
\]
and \eqref{potfe} is proved.\\
The continuity regularity follows by the Vitali Theorem. Indeed, let us consider the limit on $n\to\infty$ in the inequality in \eqref{datotp}, so that we have
\begin{equation*}
\int_{\Omega}|G_k(w(t))|^{\sigma}\,dx
\le\int_{\Omega}|G_k(u_0)|^{\sigma}\,dx+c\left(\int_0^T \||f(s)|\chi_{\{|f|>k\}}\|_{L^m(\Omega)}^r\,ds\right)^{\frac{y\beta}{r[y\beta-(\si-1)]}}
\end{equation*}
for every $k\ge k_0$. Thus, we deduce that $\int_{\Omega}|G_k(w(t))|^{\sigma}\,dx$ converges to zero if $k\to\infty$. This fact provides that 
\[
\int_{E}|w(t)|^{\sigma}\,dx\le \int_{\Omega}|G_k(w(t))|^{\sigma}\,dx+k^\sigma|E|
\]
converges to $0$ if $|E|\to 0$ and $k\to \infty$. Now, let $\{t_j\}_j$ be a sequence such that $t_j\to t$, $t\in [0,T]$, as $j\to \infty$. The continuity regularity $C([0,T];L^1(\Omega))$ proved in \cite{P} allows us to say that $w(t_j)\to w(t)$ in $L^1(\Omega)$ and conclude the proof of \eqref{ctfe}.\\

The a.e. convergence of the gradient and \eqref{A2} imply that
\[
a(t,x,u_n,\nabla u_n)\rightharpoonup a(t,x,u,\nabla u)\quad \text{weakly in}\,\,( L^{p'}(Q_T))^N
\]
and, from Proposition \ref{aeD}, we have
\[
H_n(t,x,\nabla u_n)\to H(t,x,\nabla u)\quad \text{strongly in}\,\, L^{1}(Q_T). 
\]
Moreover, we observe that, by definition of $\{u_{0,n}\}_n$, the convergence $u_{0,n}\to u_0$ in $L^{\sigma}(\Omega)$ holds.\\
Thus, we take the limit on $n$ in the weak formulation of \eqref{Pn}, so we get
\[
u\in L^p(0,T;W^{1,p}_0(\Omega))
\]
and
\begin{equation*}
-\integrale u_0(x)\vp(0,x)\,dx+\iint_{Q_T}-u\vp_t+a(t,x,u,\N u)\cdot\N\vp\,dx\,dt=\iint_{Q_T}H(t,x,\N u)\vp\,dx\,dt
\end{equation*}
for every test function $\vp$ such that
\[
\vp\in L^p(0,T;W_0^{1,p}(\Omega))\cap L^{\infty}(Q_T),\,\,\vp_t\in L^{p'}(Q_T)\text{ and }\vp(T,x)=0,
\]
otherwise, we have recovered Definition \ref{sol}.
\endproof

\section{Solutions of infinite energy} \label{secrin}
\setcounter{equation}{0}

Let us suppose that $\frac{2N}{N+2}<p<N$ and the gradient growth rate satisfying $\max\{\frac{p}{2},p-\frac{N}{N+1}\}< q<p-\frac{N}{N+2}$. In this range of $q$, the optimal conditions on the data \eqref{ID1} and \eqref{F1} do not allow us to have finite energy solutions as in Section \ref{secplap}: in particular, \eqref{ID1} implies that $1<\sigma<2$, then $u_0\in L^\sigma(\Omega)$ does not necessarily belong to $L^2(\Omega)$. This is why we are going to consider a different notion of solution.\\

We define the set of functions $\mathcal{T}^{1,p}_0(Q_T)$ as the set of all measurable functions $u:Q_T\to \mathbb{R}$ almost everywhere finite and such that the truncated functions $T_k(u)$ belong to $ L^p(0,T;W^{1,p}_0(\Omega))$ for all $k>0$. Moreover, in the spirit of \cite{BBGGPV}, we define the {generalized gradient} of a function $u$ in $\mathcal{T}^{1,p}_0(Q_T)$ as follows:
\[
\N T_k(u)=\N u \chi_{\{|u|< k\}}.
\]

\begin{definition}\label{defrin}
We say that a function $u\in \mathcal{T}^{1,p}_0(Q_T)$ is a solution of \eqref{P} if satisfies: 
\begin{equation*}
 H(t,x,\nabla u)\in L^1(Q_T), 
\end{equation*}
\begin{equation}\label{sr2}
\begin{array}{c}
\ds
-\integrale S(u_0)\vp (0,x)\,dx+\iint_{Q_T}-S(u)\vp_t+a(t,x,u,\N u)\cdot \N (S'(u)\vp)\,dx\,ds\\
[3mm]\ds
=\iint_{Q_T}H(t,x,\N u)S'(u)\vp\,dx\,ds
\end{array}
\end{equation}
for every $S\in W^{2,\infty}(\mathbb{R})$ such that $S'(\cdot)$ has compact support and for every test function $\vp\in L^p(0,T;W_0^{1,p}(\Omega))\cap  L^{\infty}(Q_T)$ such that $\vp_t\in L^{p'}(Q_T)$ and $\vp(T,x)=0$.
\end{definition}

\subsection{The a priori estimate} 

\begin{theorem}\label{sapprin}
Let $\frac{2N}{N+2}<p<N$ and assume \eqref{A1}, \eqref{A2}, \eqref{H} with $\max\{\frac{p}{2},p-\frac{N}{N+1}\}< q <p-\frac{N}{N+2}$, \eqref{F1}, \eqref{ID1} and let $\{u_n\}_n$ be a sequence of solutions of \eqref{Pn}. Then, $\{u_n\}_n$ and $\{(1+|u_n|)^{\beta-1}u_n\}_n$, $\beta=\frac{\si+p-2}{p}$, are uniformly bounded, respectively, in $L^{\infty}(0,T;L^{\sigma}(\Omega)) $ and in $L^p(0,T;W^{1,p}_0(\Omega))$. Moreover, the following estimate holds:
\begin{equation}\label{disapprin}
\sup_{t\in [0,T]}\|u_n(t)\|_{L^{\sigma}(\Omega)}^{\sigma}+\|\nabla((1+|u_n|)^{\beta-1}u_n)\|_{L^p(Q_T)}^p\le M
\end{equation}
where the constant $M$ depends on $\alpha$, $p$, $q$, $\gamma$, $N$, $|\Omega|$, $T$, $u_0$, $f$ and remains bounded when $u_0$ and $f$ vary in sets which are bounded and equi-integrable, respectively, in $L^{\sigma}(\Omega)$ and $L^r(0,T;L^m(\Omega))$.
\end{theorem}
\proof
We recall the change of variable $w_n:=e^{-t}u_n$ used in Theorem \ref{Tfe}  
and observe again that \eqref{H} and \eqref{A1}-\eqref{A2} still hold with different constants (all depending on $ T <\infty$), say $\bar \gamma$, $\bar \alpha$ and $\bar{\lambda}$.

We take  $\int_0^{G_k(w_n)}(\varepsilon+|z|)^{\sigma-3}|z|\,dz$, $\varepsilon>0$ as test function in \eqref{Pn} and integrate over $Q_t$ for $0\le t\le T$. Thus, thanks to the assumptions \eqref{A1} and \eqref{H}, we have:
\begin{multline*}
\int_{\Omega} \Theta_\varepsilon(G_k(w_n(t)))\,dx+\underbrace{\bar{\al}\iint_{Q_t}|\nabla G_k(w_n)|^p[\varepsilon+|G_k(w_n)|]^{\sigma-3}|G_k(w_n(s))| \,dx\,ds}_A\\
\le \int_{\Omega} \Theta_\vare(G_k(u_0))\,dx+\overbrace{\bar\gamma  \iint_{Q_t} |\nabla G_k(w_n)|^q\biggl(\int_0^{G_k(w_n)}(\varepsilon+|z|)^{\sigma-3}|z|\,dz\biggr)\,dx\,ds}^B\\
+\underbrace{\iint_{\{| f|>k\}\cap A_{k}\}} |f| \biggl(\int_0^{G_k(w_n)}(\varepsilon+|z|)^{\sigma-3}|z|\,dz\biggr) \,dx\,ds}_C
\end{multline*} 
where we have set $\Theta_\vare(v)=\int_0^v \left(\int_0^z (\varepsilon+|s|)^{\sigma-3}|s|\,ds \right)\,dz$, $A_{k}=A_{k,n}:=\{(s,x)\in Q_t:\,\,|w_n(s,x)|>k\}$. We also define the function $\Phi_\varepsilon(v)=\int_0^v (\varepsilon+|z|)^{\frac{\sigma-3}{p}}|z|^{\frac{1}{p}}\,dz$ so we can rewrite the $A$ term as
\[
A=\bar{\al}\iint_{Q_t}|\N \Phi_\varepsilon(G_k(w_n))|^p\,dx\,ds.
\]
Now we are going to deal with the r.h.s.. Let us start with the $B$ term. The definition of $\Phi_\varepsilon(\cdot)$ allows us to estimate as follows
\[
\begin{split}
B&\le \bar\gamma \iint_{Q_t}|\N \Phi_\varepsilon(G_k(w_n))|^q\biggl(\int_0^{G_k(w_n)}(\varepsilon+|z|)^{(\sigma-3)\frac{p-q}{p}}|z|^{\frac{p-q}{p}}\,dz\biggr)\,dx\,ds\\
&\le \bar\gamma  \iint_{Q_t} |\nabla \Phi_\varepsilon(G_k(w_n))|^q 
 |\Phi_\varepsilon(G_k(w_n))|^{p-q}|G_k(w_n)|^{q-p+1}\,dx\,ds
\end{split}
\]
where the last step is due to H\"older's inequality with indices $\left(\frac{1}{p-q},\frac{1}{q-(p-1)}\right)$ (recall that $q>p-1$). An application of the H\"older inequality with indices $\left(\frac{p}{q},\frac{p^*}{p-q},\frac{N}{p-q}\right)$, Sobolev's embedding and the definition of $\si$ (we just recall here that $\si=\frac{N(q-(p-1))}{p-q}$) give us
\begin{align*}
B\le c_1\int_0^t\|G_k(w_n(s))\|_{L^{\sigma}(\Omega)}^{\si\frac{p-q}{N}}\ \|\nabla \Phi_\varepsilon (G_k(w_n(s)))\|_{L^p(\Omega)}^p\,ds
\end{align*}
where $c_1=c_1(\bar \gamma, N,q,T)$.\\
As far as the $C$ term is concerned, we first observe that since $\si < 2$ and the equality $\si-1=\left(\frac{\si-2}{p}+1 \right)\frac{p(\si-1)}{\si+p-2}$ holds, then we have
\begin{equation}\label{ggg}
\int_0^x (\vare+|y|)^{\si-3}|y|\,dy\le c \biggl( \int_0^x  (\vare+|y|)^{\frac{\si-3}{p}}|y|^\frac{1}{p}\,dy\biggr)^{\frac{p(\sigma-1)}{\sigma+p-2}}
\end{equation}
for some $c>0$. Then, this estimate and twice applications of H\"older's inequality with $(m,m')$ and $(r,r')$ imply that we can deal with $C$ as below:
\begin{equation*}
\begin{array}{c}
\ds
C \le c_2\iint_{\{| f|>k\}\cap A_{k}}|f| \biggl(\int_0^{G_k(w_n)}(\varepsilon+|z|)^{\frac{\sigma-3}{p}}|z|^{\frac{1}{p}}\,dz\biggr)^{\frac{p(\sigma-1)}{\sigma+p-2}} \,dx\,ds\\
[3mm]\ds
\le c_2\| |f|\chi_{\{| f|>k\}}\|_{L^r(0,T;L^m(\Omega))} \|\Phi_\varepsilon(G_k(w_n))\|_{L^{r'\frac{\sigma-1}{\beta}}(0,t;L^{m'\frac{\sigma-1}{\beta}}(\Omega))}^{\frac{\sigma-1}{\beta}}
\end{array}
\end{equation*}
Then, recalling the Gagliardo-Nirenberg inequality in Theorem \ref{teoGN}, the definition of $\Phi_\vare(\cdot)$ and the assumption \eqref{F1}, we proceed as in Theorem \ref{sapp} getting
\begin{equation*}
\begin{array}{c}
\ds
C\le c_2\| |f|\chi_{\{| f|>k\}}\|_{L^r(0,T;L^m(\Omega))} \|\Phi_\varepsilon(G_k(w_n))\|_{L^{y}(0,t;L^{w}(\Omega))}^{\frac{\sigma-1}{\beta}}\\
[3mm]\ds
\le c_3\| |f|\chi_{\{| f|>k\}}\|_{L^r(0,T;L^m(\Omega))}^{\frac{y\beta}{y\beta-(\si-1)}}+c_4\|G_k(w_n)\|_{L^{\infty}(0,t;L^{\si}(\Omega))}^{\beta(y-p)}\int_0^t\|\N \Phi_\varepsilon(G_k w_n(s))\|_{L^{p}(\Omega)}^{p}\,ds.
\end{array}
\end{equation*}
We collect the estimates above saying that it holds
\[
\begin{array}{c}
\ds
\integrale \Theta_\varepsilon(G_k(w_n(t)))\,dx +\bar{\al}\int_0^t\|\N \Phi_\varepsilon(G_k w_n(s))\|_{L^{p}(\Omega)}^{p}\,ds \\[3mm]
\ds
\le \left[c_4\|G_k(w_n)\|_{L^{\infty}(0,t;L^{\si}(\Omega))}^{\beta(y-p)}+c_5\|G_k(w_n)\|_{L^{\infty}(0,t;L^{\si}(\Omega))}^{\si\frac{p-q}{N}}\right]
\int_0^t\|\nabla \Phi_\varepsilon (G_kw_n(s))\|_{L^p(\Omega)}^p\,ds 
\\ [3mm]
\ds
+c_3\| |f|\chi_{\{| f|>k\}}\|_{L^r(0,T;L^m(\Omega))}^{\frac{y\beta}{y\beta-(\si-1)}}+\integrale \Theta_\varepsilon(G_k(u_0))\,dx.
\end{array}
\]
Next, we continue reasoning as in Theorem \ref{sapp}, i.e., we fix a value $\bar{\delta}$ such that satisfies the equality\\
$ 2\max\left\{c_4 \bar{\delta}^{\frac{\beta(y-p)}{\si}}, c_5\bar{\delta}^\frac{p-q}{N} \right\}=\frac{\bar{\al}}{2}$. Furthermore, we let $\bar{k}$ large enough so that
\begin{equation}\label{00}
\|G_k(u_0)\|_{L^{\sigma}(\Omega)}^{\sigma}< \frac{\bar\delta}{2}\qquad\forall k\ge \bar{k},
\end{equation}
\begin{equation}\label{ff}
{\sigma(\sigma-1)c_3}\| |f|\chi_{\{| f|>k\}}\|_{L^r(0,T;L^m(\Omega))}^{\frac{y\beta}{y\beta-(\si-1)}}<\frac{\bar{\delta}}{2}\qquad\forall k\ge\bar{k}
\end{equation}
and, for $k\ge \bar{k}$, define $T^*$ as below:
\[
T^*:=\sup\{\tau> 0: \,||G_k(w_n(s))||_{L^{\sigma}(\Omega)}^{\sigma}\le \bar\delta, \,\,\forall   s\le \tau  \}.
\]
We notice again that $T^*>0$ due to \eqref{00} and the continuity of $w_n(t)$ in $L^\si(\Omega)$.
Then, for $ t\le T^*$ and $k\ge \bar{k}$, we have
\begin{equation}\label{ass}
\begin{array}{c}
\ds
\int_{\Omega} \Theta_\varepsilon(G_k(w_n(t)))\,dx+\frac{\bar{\al}}{2} \int_0^t\|\nabla \Phi_\varepsilon(G_k(w_n(s)))\|_{L^{p}(\Omega)}^p\,ds \\
[3mm]
\ds
\le\int_{\Omega} \Theta_\varepsilon(G_k(u_0))\,dx+c_3\| |f|\chi_{\{| f|>k\}}\|_{L^r(0,T;L^m(\Omega))}^{\frac{y\beta}{y\beta-(\si-1)}}.
\end{array}
\end{equation}
We claim that $T^*=T$. Indeed, taking $t=T^*<T$ leads to
\begin{equation*}
\int_{\Omega} \Theta_\varepsilon(G_k(w_n(T^*)))\,dx\le\\
\int_{\Omega} \Theta_\varepsilon(G_k(u_0))\,dx+c_3\|| f|\chi_{\{| f|>k\}}\|_{L^r(0,T;L^m(\Omega))}^{\frac{y\beta}{y\beta-(\si-1)}}\quad\forall k\ge\bar{k}
\end{equation*}
which, by the convergence $\Theta_\vare(G_k(w_n(s)))\underset{\vare\to 0}{\longrightarrow})\frac{|G_k(w_n(s))|^\sigma}{\sigma(\sigma-1}$, implies that
\begin{equation}\label{predeltan}
\|G_k(w_n(T^*))\|_{L^{\sigma}(\Omega)}^{\sigma}\le\\
\|G_k(u_0)\|_{L^{\sigma}(\Omega)}^{\sigma}+\sigma(\sigma-1)c_3\|| f|\chi_{\{| f|>k\}}\|_{L^r(0,T;L^m(\Omega))}^r.
\end{equation}
Thus, the conditions \eqref{00} and \eqref{ff} and the inequality in \eqref{predeltan} would give us
\[
\int_{\Omega}|G_k(w_n(T^*))|^\si\,dx<\bar{\delta}
\]
which is in contrast with the definition of $T^*$ and the continuity regularity.

Recalling that $w_n=G_k(w_n)+T_k(w_n)$, we have just proved that $\{w_n\}_n$ is bounded (uniformly in $n$) in $L^{\infty}(0,T;L^{\sigma}(\Omega))$.\\

As far as the proof of \eqref{disapprin} is concerned, we note that \eqref{ass} guarantees that $|\nabla [\Phi_\varepsilon(G_k(w_n))]|$ is uniformly bounded, in $n$ and in $\varepsilon$, in $L^p(Q_T)$.
Then, being 
\begin{align*}
\int_0^T\|\nabla \Phi_\varepsilon(G_k(w_n(s)))\|_{L^{p}(\Omega)}^p\,ds&\ge  \int_0^T\|\nabla \Phi_1(G_k(w_n(s)))\|_{L^{p}(\Omega)}^p\,ds\\
&=\iint_{Q_T} |\N G_k(w_n)|^p(1+|G_k(w_n)|)^{\sigma-3}|G_k(w_n)|\,dx\,ds\\
&\ge c\iint_{Q_T} |\N G_{k+1}(w_n)|^p(1+|G_{k+1}(w_n)|)^{\sigma-2}\,dx\,ds
\end{align*}
we get an estimate on $\|(1+|G_k(w_n(s))|)^{\beta-1}G_k(w_n(s))\|_{L^p(Q_T)}$ for $k\ge \bar k+1$.\\

Finally, since
\[
\begin{array}{c}
\ds
\iint_{Q_T}|\nabla ((1+|w_n|)^{\beta-1}w_n)|^p\,dx\,dt\le c\iint_{Q_T}\frac{|\N w_n|^p}{(1+|w_n|)^{p(1-\beta)}}\,dx\,dt\\
[3mm]\ds
\le c\iint_{Q_T\cap \{|w_n|>k\}}\frac{|\N G_k(w_n)|^p}{(1+|G_k(w_n)|)^{p(1-\beta)}}\,dx\,dt+c\iint_{Q_T\cap\{|w_n|\le k\}}|\N T_k(w_n)|^p\,dx\,dt
\end{array}
\]
the inequality \eqref{disapprin} follows taking $T_k(w_n)$ as test function and reasoning as in the second part of Theorem \ref{sapp}.
\endproof

\subsection{Some convergence results}
\begin{proposition}\label{abc}
Let $\frac{2N}{N+2}<p<N$ and assume \eqref{A1}, \eqref{A2}, \eqref{A3}, \eqref{H} with $\max\{\frac{p}{2},p-\frac{N}{N+1}\}< q <p-\frac{N}{N+2}$, \eqref{F1}, \eqref{ID1} and let $\{u_n\}_n$ be a sequence of solutions of \eqref{Pn}. Then, for some function $u$, we have that 
\[
u_n\to u\quad\text{a.e. in }\,\,Q_T,
\]
\[
\N u_n\to\N u\quad \text{a.e. in }\,\,Q_T,
\] 
\[
H_n(t,x,\nabla u_n)\to H_n(t,x,\nabla u)\quad\text{strongly in}\quad L^1(Q_T)
\]
and
\[
T_k(u_n)\to T_k(u)\quad\text{strongly in}\quad L^p(0,T;W_0^{1,p}(\Omega))\quad\forall  k>0.
\]
\end{proposition}
\proof 
\textit{The boundedness of $\{|\N u_n|^\eta \}_n$.}\\
First of all, we prove that $\{|\N u_n|^\eta\}_n$ is uniformly bounded in $L^1(Q_T)$ for some $q<\eta<p$. Indeed, this fact will allow us to reason as in \cite[Theorem $3.3$]{BDGO} (being the r.h.s. uniformly bounded in $L^1(Q_T)$) and get the a.e. convergence of the gradient.

Since $(1+|u_n|)^{\beta-1}u_n<(1+|u_n|)^{\beta}$ for every $n\in \mathbb{N}$, Theorem \ref{sapprin} implies that $\{(1+|u_n|)^{\beta-1}u_n\}_n$ is uniformly bounded in $ L^\infty(0,T;L^{\frac{\sigma}{\beta}}(\Omega))\cap  L^p(0,T;W_0^{1,p}(\Omega))$. Note that $\frac{\si}{\beta}<p^*$ if and only if $p>\frac{2N}{N+\si}$ which is $q>\frac{p}{2}$: we thus apply Gagliardo-Nirenberg regularity results and deduce that $\{(1+|u_n|)^{\beta-1}u_n\}_n$ is bounded in $ L^{p\frac{N\beta+\sigma}{N\beta}}(Q_T)$.
Moreover, being $\{(1+|u_n|)^{\beta-1}u_n\}_n$ bounded in $L^p(0,T;W_0^{1,p}(\Omega))$ too, we have that
\begin{equation}\label{bbb}
\begin{split}
c>&\iint_{Q_T}|\N [(1+|u_n|)^{\beta-1}|u_n| ]|^p\,dx\,dt\\
&=\iint_{Q_T}\frac{ |\N u_n|^p}{(1+|u_n|)^{p(2-\beta)}}(1+\beta |u_n|)^p\,dx\,dt\\
&> \beta^p\iint_{Q_T}\frac{ |\N u_n|^p}{(1+|u_n|)^{p(1-\beta)}}\,dx\,dt.
\end{split}
\end{equation}
We now look for a bound for a suitable power of the gradient, that is, we employ \eqref{bbb} so that
\begin{align*}
\iint_{Q_T}|\nabla u_n|^\eta\,dx\,ds&= \iint_{Q_T}\frac{|\nabla u_n|^\eta}{(1+|u_n|)^{\eta(1-\beta)}}(1+|u_n|)^{\eta(1-\beta)} \,dx\,dt\\
&\le \biggl(\iint_{Q_T} \frac{|\nabla u_n|^p}{(1+|u_n|)^{p(1-\beta)}}\,dx\,dt\biggr)^{\frac{\eta}{p}}\biggl(\iint_{Q_T} (1+|u_n|)^{p\eta\frac{(1-\beta)}{p-\eta}}\,dx\,dt\biggr)^{\frac{p-\eta}{p}}\\
&\le c.
\end{align*}
Hence, we have to choose $\eta$ such that $\ds p\eta\frac{(1-\beta)}{p-\eta}=p\frac{N\beta+\sigma}{N}$. This condition leads to $\ds\eta=p\frac{N\beta+\sigma}{ N+\sigma}=N(q-(p-1))+2q-p$ and thus it holds $q< p\frac{N\beta+\sigma}{N+\sigma}<p$ since $\beta<1$
and $\sigma> 1$. \\

\noindent
\textit{The a.e. convergence $u_n\to u$.}\\
So far, we have that $\{|\N u_n|^\eta\}_n$ is bounded in $L^1(Q_T)$. In particular, if $p>2-\frac{1}{N+1}$, then $\{|\N u_n|\}_n$ is bounded in $L^\eta(Q_T)$. This means that we can invoke \cite[Corollary $4$]{S} with $X=W^{1,\eta}_0(\Omega)$, $B=L^\eta(Q_T)$ and $Y=W^{-1,s'}(\Omega)$ with $s>N$ obtaining that $\{u_n\}_n$ is compact in $L^\eta(Q_T)$. We thus deduce that, up to subsequences, $u_n\to u$ a.e..

If, otherwise, $\frac{2N}{N+2}<p\le 2-\frac{1}{N+1}$, we reason as in \cite[Theorem $2.1$]{P}. In particular, we point out that \cite[Corollary $4$]{S} is applied (as in Corollary \ref{simon}) to a regularization of the truncation function $T_k(u_n)$ instead of to $u_n$ itself. \\

\noindent
\textit{The a.e. convergence $\N u_n\to \N u$.}\\
Since the r.h.s. is bounded in $L^1(Q_T)$ and $u_n\to u$ a.e., the  a.e. convergence of the gradients follows from \cite{BDGO}.\\

\noindent
\textit{The strong convergence $H_n(t,x,\N u_n)\to H_n(t,x,\N u)$ in $L^1(Q_T)$.}\\
We conclude saying that, since $\eta>q$, we get the equi-integrability of the r.h.s. $H_n(t,x,\N u_n)$ in $L^1(Q_T)$ and so an application of the Vitali Theorem gives us the desired convergence as well.\\

\noindent
\textit{The strong convergence $T_k(u_n)\to T_k(u)$ in $L^p(0,T;W_0^{1,p}(\Omega))$.}\\
The convergence in $L^1(Q_T)$ of the r.h.s. allows us to reason as in \cite{BlMu,P}, getting the strong convergence of the truncation functions.
\endproof

\subsection{The existence result}
\begin{theorem}\label{Trin}
Let $\frac{2N}{N+2}<p<N$ and assume \eqref{A1}, \eqref{A2}, \eqref{A3}, \eqref{H} with $\max\{\frac{p}{2},p-\frac{N}{N+1}\}< q <p-\frac{N}{N+2}$, \eqref{F1} and \eqref{ID1}. 
Then, there exists at least one solution of the problem \eqref{P} in the sense of Definition \ref{defrin} satisfying
\begin{equation}\label{reg}
(1+|u|)^{\beta-1}u\in L^p(0,T;W_0^{1,p}(\Omega))\quad\text{with}\quad \beta=\frac{\sigma-2+p}{p}
\end{equation}
and
\begin{equation}\label{ct}
u\in C([0,T];L^\sigma(\Omega)).
\end{equation}
\end{theorem}
\proof
Let $\{u_n\}_n\subseteq L^p(0,T;W^{1,p}_0(\Omega))$ be a sequence of weak solutions of \eqref{Pn}. \\
The uniform bound \eqref{disapprin} and the a.e. convergences $u_n\to u$ and $\N u_n\to \N u$ imply that $u$ satisfies the following inequality:
\begin{equation*}
\sup_{t\in [0,T]}\|u(t)\|_{L^{\sigma}(\Omega)}^{\sigma}+\|\nabla((1+|u|)^{\beta-1}u)\|_{L^p(Q_T)}^p\le c.
\end{equation*}
This means that $u\in L^\infty(0,T;L^{\sigma}(\Omega))$ and $(1+|u|)^{\beta-1}u\in L^p(0,T;W_0^{1,p}(\Omega))$.

The continuity regularity \eqref{ct} is a consequence of the Vitali Theorem and can be proved as in Theorem \ref{Tfe} (taking into account the inequality \eqref{predeltan}).

Now, we focus on \eqref{sr2}.
We multiply the equation in \eqref{Pn} for $S'(u_n)\vp$ and integrate over $Q_T$, getting
\begin{equation}\label{fdn}
\begin{array}{c}
\ds
-\integrale S(u_n(0))\vp (0,x)\,dx+\iint_{Q_T}-S(u_n)\vp_t\,dx\,ds
+\iint_{Q_T}S'(u_n)a(t,x,u_n,\N u_n)\cdot \N \vp\,dx\,ds
\\
[3mm]\ds
+\iint_{Q_T}S''(u_n)a(t,x,u_n,\N u_n)\cdot \N u_n\vp\,dx\,ds
=\iint_{Q_T}H(t,x,\N u_n)S'(u_n)\vp\,dx\,ds
\end{array}
\end{equation}
where $\vp\in L^p(0,T;W_0^{1,p}(\Omega))\cap L^\infty(Q_T)$, $\vp_t\in L^{p'}(Q_T)$ and $\vp(T,x)=0$.
The proof will be concluded once we show that the limit on $n\to \infty$ can be taken in \eqref{fdn}. Note that, being $\text{supp}(S'(u_n))\subseteq [-M,M]$, the equation \eqref{fdn} takes into account only $T_M(u_n)$.

The previous remark and Proposition \ref{abc} imply that
\begin{multline}\label{convg1}
S'(u_n)a(t,x,u_n,\nabla u_n)=S'(u_n)a(t,x,T_M(u_n),\nabla T_M(u_n))\\
\to S'(u)a(t,x,T_M(u),\nabla T_M(u))\quad\text{in}\,\,(L^{p'}(Q_T))^N.
\end{multline}

Moreover, since
\[
a(t,x,u_n,\nabla u_n)\cdot\nabla u_n=a(t,x,T_M(u_n),\nabla T_M(u_n)\cdot\nabla T_M(u_n).
\]
the strong convergence of $T_M(u_n)\to T_M(u)$ in  $L^p(0,T;W^{1,p}_0(\Omega))$ implies that
\[
a(t,x,T_M(u_n),\nabla T_M(u_n))\cdot\nabla T_M(u_n)
\to a(t,x,T_M(u),\nabla T_M(u))\cdot\nabla T_M(u) \qquad\text{in}\,\,L^1(Q_T).
\]
Having $S''(u_n)\to S''(u)$ pointwisely and being $S''(\cdot)$ bounded by assumptions give us the following convergence:
\begin{multline}\label{convg2}
S''(u_n)a(t,x,T_M(u_n),\nabla T_M(u_n))\cdot\nabla T_M(u_n)
\to S''(u)a(t,x,T_M(u),\nabla T_M(u))\cdot\nabla T_M(u) \qquad\text{in}\,\,L^1(Q_T).
\end{multline}
\indent
As far as the r.h.s. is concerned, Proposition \ref{abc} guarantees that
\begin{equation}\label{convg3}
S'(u_n)H_n(t,x,\N u_n)=S'(u_n)H_n(t,x,\N T_M(u_n))
\to S'(u)H(t,x,\N T_M(u))\qquad\text{in}\,\, L^1(Q_T).
\end{equation}


Finally, since
\[
S(u_n)\to S(u)\qquad\text{ in}\quad L^p(0,T;W^{1,p}(\Omega)),
\]
\[
(S(u_n))_t\to (S(u))_t\qquad\text{ in}\quad L^{p'}(0,T;W^{-1,p'}(\Omega))+L^1(Q_T)
\]
and thus $S(u_n)\to S(u)$ strongly in $C([0,T];L^1(\Omega))$ thanks to \cite[Theorem $1$]{P}, we can take the limit in \eqref{fdn} finding Definition \ref{defrin}.
\endproof

\begin{remark}\label{secq}

We briefly present the case when $\frac{2N}{N+1}<p<N$ and the growth rate of the gradient is $q=p-\frac{N}{N+1}$ and the initial datum is taken in the Lebesgue space $L^{1+\omega}(\Omega)$ for $\omega\in(0,1)$ and we are looking for renormalized solutions in the sense of Definition \ref{defrin}. Note that this value of $q$ implies that $\sigma=1$, so our running assumptions are not the sharp ones.\\
However, having a stronger regularity on the initial datum allows us to repeat the proofs presented in Section \ref{secrin}. In particular, the a priori estimate reads as below:
\begin{theorem}\label{sapprinc}
Assume \eqref{A1}, \eqref{A2}, \eqref{H} with $q=p-\frac{N}{N+1}$, $u_0\in L^{1+\omega}(\Omega)$, $\omega>0$, \eqref{F2} and let $\{u_n\}_n$ be a sequence of solutions of \eqref{Pn}. Then, $\{u_n\}_n$ and $\{(1+|u_n|)^{\nu-1}u_n\}_n$, $\nu=\frac{p-1+\omega}{p}$, are uniformly bounded, respectively, in $L^{\infty}(0,T;L^{1+\omega}(\Omega)) $ and in $L^p(0,T;W^{1,p}_0(\Omega))$. Moreover, the following estimate holds:
\begin{equation}\label{disapprinc}
\sup_{t\in [0,T]}\|u_n(t)\|_{L^{1+\omega}(\Omega)}^{1+\omega}+\|\nabla((1+|u_n|)^{\nu-1}u_n)\|_{L^p(Q_T)}^p\le M 
\end{equation}
where the constant $M$ depends on $\alpha$, $p$, $q$, $\gamma$, $N$, $|\Omega|$, $T$, $\omega$, $u_0$, $f$ and remains bounded when $u_0$ and $f$ vary in sets which are, respectively, bounded in $L^{1+\omega}(\Omega)$ and equi-integrable in $L^1(Q_T)$.
\end{theorem}

We observe that the case $q=p-\frac{N}{N+1}$ could be dealt with taking the initial datum $u_0$ in the Orlicz space
\[
u_0\in L^1 ((\log L)^1).
\]

\end{remark}

\section{Case of $L^1$ data}\label{secrin2}
\setcounter{equation}{0}

We finally take into account the case with $\frac{2N}{N+1}<p<N$ and $\max\{\frac{p}{2},\frac{p(N+1)-N}{N+2}\}<q< p-\frac{N}{N+1}$.\\
These ranges of $q$ imply that the value $\sigma$ (defined in \eqref{ID1}) is smaller than one. In this range even measure data could be considered, however, we focus on $L^1(\Omega)$ data for the sake of simplicity. We go further introducing our current notion of renormalized solution. 

\begin{definition}\label{defrin2}
We say that a function $u\in \mathcal{T}_0^{1,p}(Q_T)$ is a renormalized solution of \eqref{P} if satisfies
\begin{equation*}
 H(t,x,\nabla u)\in L^1(Q_T),
\end{equation*}
\begin{equation}\label{sr22}\tag{RS.1}
\begin{array}{c}
\ds
-\integrale S(u_0)\vp (0,x)\,dx+\iint_{Q_T}-S(u)\vp_t+a(t,x,u,\N u)\cdot \N (S'(u)\vp)\,dx\,ds\\
[3mm]\ds
=\iint_{Q_T}H(t,x,\N u)S'(u)\vp\,dx\,ds
\end{array}
\end{equation}
\begin{equation}\label{sr32}\tag{RS.2}
\lim_{n\to \infty}\frac{1}{n}\iint_{\{n\le |u|\le 2n\}}a(t,x,u,\N u)\cdot \N u=0, 
\end{equation}

for every $S\in W^{2,\infty}(\mathbb{R})$ such that $S'$ has compact support and for every test function $\vp\in L^p(0,T;W_0^{1,p}(\Omega))\cap  L^{\infty}(Q_T)$ such that $\vp_t\in L^{p'}(Q_T)$ and $\vp(T,x)=0$.
\end{definition}
\begin{remark}
Note that the main difference between Definitions \ref{defrin} and \ref{defrin2} relies in the condition \eqref{sr32}. Indeed, the setting considered in Section \ref{secrin} ensures that a solution in the sense of Section \ref{defrin} enjoys
\begin{equation*}
(1+|u|)^{\beta-1}u\in L^p(0,T;W^{1,p}_0(\Omega)), \quad  \beta=\frac{\sigma+p-2}{p}<1,
\end{equation*}
which implies \eqref{sr32}. 
Roughly speaking, \eqref{sr32} regards the behaviour for "large" values of $u$ (i.e., the case $|u|=\infty$ which is excluded by the truncated function) and it is a standard request in the renormalized framework. \\
For further comments on the notion of renormalized solution we mention \cite{LM,BGDM,DMOP} for the stationary setting and \cite{BlMu,BlP}  for what concerns the evolution framework.
\end{remark}

We will need different spaces from the Lebesgue and the Sobolev's ones we have used so far. In particular, we will use the \emph{Marcinkievicz} space of $\gamma$ order $M^\gamma(Q_T)$. So, let us recall the definition and a few properties of this space. \\
Let $0<\gamma<\infty$. Then $M^{\gamma}(Q_T)$ is defined as the set of measurable functions $f:Q_T\to \mathbb{R}$ such that 
\[
\meas\{(t,x)\in Q_T:\,\,|f|>k \}\le ck^{-\gamma}
\]
and it is equipped with the norm
\[
\|f\|_{M^{\gamma}(Q_T)}=
\sup_{k>0}\left\{k^\gamma \meas\{(t,x)\in Q_T:\,\,|f|>k \}\right\}^{\frac{1}{\gamma}}.
\] 
Moreover, the following embeddings hold
\[
L^\gamma(Q_T)\hookrightarrow M^{\gamma}(Q_T)\hookrightarrow L^{\gamma-\omega}(Q_T)
\]
for every $\omega\in (0,\gamma-1]$, $\gamma> 1$.
\subsection{The a priori estimate and convergence results}

\begin{theorem}\label{sapp3}
Let $\frac{2N}{N+1}<p<N$ and assume \eqref{A1}, \eqref{A2}, \eqref{H} with $\frac{p(N+1)-N}{N+2}<q< p-\frac{N}{N+1}$, \eqref{F2}, \eqref{ID2} and let $\{u_n\}_n$ be a sequence of solutions of \eqref{Pn}. Then $\{ u_n \}_n$ is uniformly bounded in $L^\infty(0,T;L^1(\Omega))$ and $\{|\N u_n|^q\}_n$ is uniformly bounded in $L^1(Q_T)$.
%
%
\end{theorem}
\begin{proof}
We first prove that we are in the framework of Lemma \ref{MGN}.\\
With this purpose, we take $\left(1-\frac{1}{(1+|G_k(w_n)|)^{\delta}} \right)\text{sign}(u_n)$, $\delta>1$ arbitrary large, as test function. Thus, dropping the gradient term in the l.h.s., we have
\begin{equation*}
\begin{array}{c}
\ds\integrale |G_k(w_n(t))|\,dx-\frac{1}{\delta-1}\integrale\left( 1- \frac{1}{(1+|G_k(w_n(t))|)^{\delta-1}}\right)\,dx\\
[3mm]\ds
\le \integrale |G_k(u_0)|\,dx-\frac{1}{\delta-1}\integrale\left( 1- \frac{1}{(1+|G_k(u_0)|)^{\delta-1}}\right)\,dx
+\gamma\iint_{Q_T}|\N G_k(w_n)|^q \,dx\,dt
 +\iint_{Q_T}|f|\chi_{\{|f|>k\}}\,dx\,dt.
\end{array}
\end{equation*}
Letting $\delta\to \infty$, we obtain
\begin{align*}
\integrale |G_k(w_n(t))&|\,dx
\le \integrale |G_k(u_0)|\,dx+\gamma\iint_{Q_T}|\N G_k(w_n)|^q \,dx\,dt
 +\iint_{Q_T}|f|\chi_{\{|f|>k\}}\,dx\,dt.
\end{align*}
This means that $G_k(w_n)$ satisfies an inequality of the type $\displaystyle\|G_k(w_n)\|_{L^\infty(0,T;L^1(\Omega))}\le M$ where
\[
M=\gamma\iint_{Q_T}|\N G_k(w_n)|^q \,dx\,dt
 +\iint_{Q_T}|f|\chi_{\{|f|>k\}}\,dx\,dt+\integrale |G_k(u_0)|\,dx.
\]

If, instead, we take $T_j(G_k(w_n))$ as test function then we can estimate as below:
\begin{equation}\label{disz}
\begin{array}{c}
\ds
\al \iint_{Q_T} |\N T_j(G_k(w_n))|^p\,dx\,dt\\
[3mm]
\ds
\le j\biggl[\gamma \iint_{Q_T}|\N G_k(w_n)|^q \,dx\,dt
+\iint_{Q_T}|f|\chi_{\{|f|>k\}}\,dx\,dt+\integrale |G_k(u_0)|\,dx\biggr]
\end{array}
\end{equation}
and, in particular, we deduce that $\ds\al\iint_{Q_T} |\N T_j(G_k(w_n))|^p\,dx\,ds\le jM$.

Thus we can apply Lemma \ref{MGN} with $v=G_k(w_n)$ obtaining
\[
\begin{array}{c}
\ds
\||\N G_k(w_n)|^{\frac{p(N+1)-N}{N+2}}\|_{M^{\frac{N+2}{N+1}}(Q_T)}\\
[3mm]\ds
\le c\left[ \||\N G_k(w_n)|^{\frac{p(N+1)-N}{N+2}}\|_{L^{q\frac{N+2}{p(N+1)-N}}(Q_T)}^{q\frac{N+2}{p(N+1)-N}} +\|f|\chi_{\{|f|>k\}}\|_{L^1(Q_T)}+\|G_k(u_0)\|_{L^1(\Omega)}  \right]. 
\end{array}
\]
Our current assumptions ensure that $\displaystyle q\frac{N+2}{p(N+1)-N}<\frac{N+2}{N+1}$ and so we have that the embedding  $\displaystyle M^{\frac{N+2}{N+1}}(Q_T)\subset L^{q\frac{N+2}{p(N+1)-N}}(Q_T)$ holds. We thus go further estimating from below as follows: 
\begin{equation*}
\begin{array}{c}
\ds
\||\N G_k(w_n)|^{\frac{p(N+1)-N}{N+2}}\|_{L^{q\frac{N+2}{p(N+1)-N}}(Q_T)}\\[3mm]\ds
\le c\biggl[ \||\N G_k(w_n)|^{\frac{p(N+1)-N}{N+2}}\|_{L^{q\frac{N+2}{p(N+1)-N}}(Q_T)}^{q\frac{N+2}{p(N+1)-N}} +\|f|\chi_{\{|f|>k\}}\|_{L^1(Q_T)}+\|G_k(u_0)\|_{L^1(\Omega)}  \biggr].
\end{array}
\end{equation*}
Now, set $Y_{n,k}=\||\N G_k(w_n)|^{\frac{p(N+1)-N}{N+2}}\|_{L^{q\frac{N+2}{p(N+1)-N}}(Q_T)}$ and $h_k=\|f\chi_{\{|f|>k\}}\|_{L^1(Q_T)}+\|G_k(u_0)\|_{L^1(\Omega)}$ so we rewrite
\[
Y_{n,k}-c_1 Y_{n,k}^{\frac{q(N+2)}{p(N+1)-N}}\le c_2h_k.
\]

We thus can reason as in \cite{GMP}, otherwise, we define the function
\[
F(Y)=Y-c_1Y^{\frac{q(N+2)}{p(N+1)-N}}
\]
which has a unique maximizer $Z^*=\left(  \frac{p(N+1)-N}{c_1q(N+2)}  \right)^{\frac{p(N+1)-N}{(N+1)(q-p+1)+q-1}}$. In particular, $F(Z^*)=F^*=F^*(p,N,q,\al,\gamma)$.
Coming back to the inequality
\[
F(Y_{n,k})\le c_2h_k
\]
we observe that it is not trivial only if $c_2h_k<F^*$. Hence, taking in mind the definition of $h_k$, we define
\[
k^*=\inf\{k>0:\,\, c_2h_k<F^* \}.
\]
Such a value of $k^*$ ensures that, being $h_k$ non increasing in $k$, we have that $c_2h_k<F^*$ for every $k\ge k^*$. We now consider the equation $F(Y_{n,k})=c_2h_k$ and observe that it admits two roots, say $Z_1$ and $Z_2$, which satisfy $0\le Z_1<Z^*<Z_2$. Thus the inequality $F(Y_{n,k})\le c_2h_k$ implies that either $Y_{n,k}\le Z_1$ or $Y_{n,k}\ge Z_2$. Since the continuity of the function $k\to Y_{n,k}$ and the convergence to zero of $Y_{n,k}$ for $k\to \infty$ imply that $Y_{n,k}\le Z_1$ for all $k\ge k^*$  we can say that
\begin{equation}\label{kstar}
\iint_{Q_T} |\N G_k(w_n)|^q\,dx\,dt\le c\qquad\forall  k\ge k^*.
\end{equation}
Finally, being 
\[
\iint_{Q_T}|\N u_n|^q\,dx\,dt\le c\left[\iint_{Q_T} |\N G_{k^*}(w_n)|^q\,dx\,dt+\iint_{Q_T} |\N T_{k^*}(w_n)|^q\,dx\,dt\right]
\]
we need an estimate on the last integral in order to prove that $\{|\N u_n|^q\}_n$ is uniformly bounded in $L^1(Q_T)$.
We take $T_{k^*}(u_n)$ as test function, obtaining
\begin{equation*}
\al \iint_{Q_T} |\N T_{k^*}(u_n)|^p\,dx\,dt
\le k^*\biggl[\gamma \iint_{Q_T}|\N u_n|^q \,dx\,dt
+\iint_{Q_T}|f|\,dx\,dt+\integrale |u_0|\,dx\biggr].
\end{equation*}
from which
\begin{equation*}
\begin{array}{c}
\ds\iint_{Q_T} |\N T_{k^*}(w_n)|^q\,dx\,dt\le c\left[k^*\biggl( \iint_{Q_T}|\N u_n|^q \,dx\,dt
+\iint_{Q_T}|f|\,dx\,dt+\integrale |u_0|\,dx\biggr)\right]^{\frac{q}{p}}\\
[3mm]\ds\le c\left(k^*\iint_{Q_T} |\N T_{k^*}(w_n)|^q\,dx\,dt \right)^{\frac{q}{p}}+c
\end{array}
\end{equation*}
thanks to \eqref{kstar}. Young's inequality allows us to conclude saying that
\[
\iint_{Q_T} |\N T_{k^*}(w_n)|^q\,dx\,dt\le c
\]
where $c$ depends on $k^*$, above the parameters of the problem.
\end{proof}

\begin{proposition}\label{abc3}
Let $\frac{2N}{N+1}<p<N$ and assume \eqref{A1}, \eqref{A2}, \eqref{A3}, \eqref{H} with $\frac{p(N+1)-N}{N+2}<q< p-\frac{N}{N+1}$, \eqref{F2}, \eqref{ID2} and let $\{u_n\}_n$ be a sequence of solutions of \eqref{Pn}. Then, for some function $u$,
\begin{equation}\label{abis}
 u_n\to u\quad\text{a.e. in } Q_T,
\end{equation}
\begin{equation}\label{a}
\N u_n\to\N u\quad\text{a.e. in } Q_T,
\end{equation}
\begin{equation}\label{b}
H_n(t,x,\nabla u_n)\to H(t,x,\nabla u)\quad\text{ in  }L^1(Q_T),
\end{equation}
\begin{equation}\label{c}
T_k(u_n)\to T_k(u)\quad\text{ in  }L^p(0,T;W_0^{1,p}(\Omega))\,\,\forall  k>0.
\end{equation}
\end{proposition}
\begin{proof}
So far, we know that the r.h.s. $\{H_n(t,x,\nabla u_n)\}_n$ is bounded in $L^1(Q_T)$. Classical estimates (see \cite{BG,BDGO}) allow us to get the a.e. convergence \eqref{abis} (see also Proposition \ref{abc}) and the boundedness of $\{|\N u_n|\}_n$ in $M^{\frac{p(N+1)-N}{N+1}}(Q_T)$. 

The a.e. convergence in \eqref{abis} and the boundedness of the r.h.s. in $L^1(Q_T)$ give us the a.e. convergence of the gradients \eqref{a} thanks to \cite{BDGO}.

Moreover, being $q<\frac{p(N+1)-N}{N+1}$, the boundedness of $\{|\N u_n|\}_n$ in $M^{\frac{p(N+1)-N}{N+1}}(Q_T)$ implies the equi-integrability of the r.h.s. and thus the strong convergence in $L^1(Q_T)$ of the r.h.s. through an application of the Vitali Theorem.

Finally, we deduce the strong convergence of the truncation \eqref{c} recalling \cite{BlMu,P}.
\end{proof}

\subsection{The existence result}
\begin{theorem}\label{Trin2}
Let $\frac{2N}{N+1}<p<N$ and assume \eqref{A1}, \eqref{A2}, \eqref{A3}, \eqref{H} with $\frac{p(N+1)-N}{N+2}<q< p-\frac{N}{N+1}$, \eqref{F2} and \eqref{ID2}. 
Then, there exists at least one renormalized solution of the problem \eqref{P} in the sense of Definition \ref{defrin2} satisfying
\begin{equation*}
u\in C([0,T];L^1(\Omega)).
\end{equation*}
\end{theorem}
\proof
The renormalized formulation \eqref{sr22} can be proved reasoning as in Theorem \ref{Trin}; the continuity regularity $C([0,T];L^1(\Omega))$ can be deduced recalling the trace result \cite{P} as well.\\
Finally, the energy growth conditions \eqref{sr32} is a consequence of the proof of \cite[Theorem $2$]{Bl} (see also \cite[Lemma $3.2$ and Remark $2.4$]{BlMu}).
\endproof

\section{On the sublinear problem}\label{secsub}
\setcounter{equation}{0}

We are going to briefly discuss what we called the sublinear case. As we have already mentioned, this case was previously analysed in \cite{Po,DNFG}. More precisely, the authors take into account a parabolic problem of Cauchy-Dirichlet type with lower order terms which grow as a power of the gradient $\ds|\N u|^q$ for $\ds q\le \frac{N(p-1)+p}{N+2}$. We have already pointed out that such a threshold is not sharp for $1<p<2$ since the borderline for the superlinear growth becomes $q=\frac{p}{2}$. We refer to Section \ref{subsecsharp} for further details on the argument presented to justify this assertion.

So, let us our parabolic Cauchy-Dirichlet problem:
\begin{equation}\label{Psub}\tag{$P_{\text{sub}}$}
\begin{cases}
\begin{array}{ll}
u_t-\dive  a(t,x,u,\nabla u) =H(t,x,\nabla u) & \text{in}\,\,Q_T,\\
 u=0  &\text{on}\,\,(0,T)\times \partial \Omega,\\
 u(0,x)=u_0(x) &\text{in}\,\, \Omega,
\end{array}
\end{cases}
\end{equation}
where $p$ is assumed to be $1<p<2$. The hypotheses \eqref{Psub} are listed below:
\begin{itemize}
\item the Leray-Lions structure conditions \eqref{A1}, \eqref{A2}, \eqref{A3} hold;
\item the function $H(t,x,\xi):(0,T)\times \Omega\times \mathbb{R}^N\to \mathbb{R}$ grows at most as a power of the gradient plus a forcing term, namely
\begin{equation}\label{Hsub}\tag{$H_{\text{sub}}$}
\ds
\begin{array}{c}
\exists \,\,\gamma\,\,\text{s.t. }\,\,
|H(t,x,\xi)|\le \gamma |\xi|^q+f(t,x)\quad
 \text{with}\quad0<q\le\frac{p}{2},
\end{array}
\end{equation}
a.e. $(t,x)\in Q_T$, for all $\xi\in\mathbb{R}^N$ and with $f=f(t,x)$ is some Lebesgue space.
\end{itemize}
As far as the data are concerned, we assume that the initial datum verifies
\begin{equation}\label{id1}\tag{$ID1_{sub}$}
u_0\in L^m(\Omega)\quad
 \text{with}\quad m> 1 \quad \text{if}\quad 0<q\le \frac{p}{2}
\end{equation}
and 
\begin{equation}\label{id2}\tag{$ID2_{sub}$}
u_0\in L^m(\Omega)\quad
 \text{with}\quad m\ge 1\quad \text{if}\quad 0<q< \frac{p}{2}
\end{equation}
and the forcing term satisfies
\begin{equation}\label{Fsub}\tag{$F_{\text{sub}}$}
f\in L^1(0,T;L^m(\Omega)).
\end{equation}
We note here that, if $m>1$, we could only deal with the linear case $q=\frac{p}{2}$ since, by Young's inequality, $|\N u|^q\le |\N u|^{\frac{p}{2}}+c$ when $q<\frac{p}{2}$. However, we will separate the growths $q<\frac{p}{2}$ and $q=\frac{p}{2}$ in order to stress the features of the \emph{sublinear} and \emph{linear} settings.

\subsection{The a priori estimate}
\begin{theorem}\label{sappsub}
Assume $1<p<2$, \eqref{A1}, \eqref{A2}, \eqref{Hsub}, either \eqref{id1} or \eqref{id2}, \eqref{Fsub} and let $\{u_n\}_n$ be a sequence of solutions of \eqref{Pn}. Then, $\{u_n\}_n$ is uniformly bounded in   $L^\infty(0,T;L^{m}(\Omega))$.
Moreover, we have that:
\begin{itemize}
\item if $m>1$, then $\{(1+|u_n|)^{\mu-1}u_n\}_n$ is uniformly bounded in $L^p(0,T;W^{1,p}_0(\Omega))$ and the following estimate holds:
\begin{equation}\label{disappsub1}
\sup_{t\in [0,T]}\|u_n(t)\|_{L^{m}(\Omega)}^{m}+\|\nabla((1+|u_n|)^{\mu-1}u_n)\|_{L^p(Q_T)}^p\le M,\qquad \mu=\frac{m+p-2}{p}.
\end{equation}
In particular, if $m\ge 2$, then $\{u_n\}_n$ is uniformly bounded in $L^p(0,T;W^{1,p}_0(\Omega))$.\\
The constant $M$ above depends on $\alpha$, $p$, $q$, $\gamma$, $N$, $|\Omega|$, $T$, $\|f\|_{L^1(0,T;L^m(\Omega))}$ and $\|u_0\|_{L^m(\Omega)}$.
\item If $m=1$ and $q<\frac{p}{2}$, then $\{u_n\}_n$ satisfies the estimates of Lemmas \ref{MGN} and \ref{MC}.
\end{itemize}
\end{theorem}
\begin{remark}
Note that the constant $M$ depends on the initial datum $u_0$ and on the forcing term $f$ through their norms an \emph{not} through an \emph{equi-integrability relation}. We recall that this fact is due to the sublinear behaviour of the r.h.s..  
\end{remark}
\begin{proof}
\textit{The case $m\ge 2$ and $q<\frac{p}{2}$.}\\
Having $m\ge2$ allows us to multiply the equation in \eqref{Pn} by $\vp(u_n)=|u_n|^{m-2}u_n$ so that an integration over $Q_t$ gives us
\begin{equation}\label{mq}
\begin{array}{c}
\ds
\frac{1}{m}\integrale |u_n(t)|^{m}\,dx+\frac{\al(m-1)}{\mu^p}\iint_{Q_t}|\N [|u_n|^{\mu}]|^p\,dx\,ds\\
[3mm]\ds
\le\underbrace{ \gamma\iint_{Q_t}|\N u_n|^q |u_n|^{m-1}\,dx\,ds}_A +\underbrace{\iint_{Q_t}|f\|u_n|^{m-1}\,dx\,ds}_B+\frac{1}{m}\integrale |u_0|^{m}\,dx.
\end{array}
\end{equation}
Dealing with the $A$ term turns out to be simpler than before, since the sublinear growth guarantees that we can proceed estimating by Young's inequality with indices $\left(\frac{p}{q},\frac{p}{p-q}\right)$ as below
\begin{align*}
A&=\frac{\gamma}{\mu^q}\iint_{Q_t}|\N [|u_n|^{\mu}]|^q|u_n|^{(m-1)\frac{p-q}{p}+\frac{q}{p}}\,dx\,ds\\
&\le \frac{\al(m-1)}{2\mu^p}\iint_{Q_t}|\N [|u_n|^{\mu}]|^p\,dx\,ds+c_1\iint_{Q_t}|u_n|^{m-1+\frac{q}{p-q}}\,dx\,ds.
\end{align*}
We point out that having $q<\frac{p}{2}$ implies that the exponent $m-1+\frac{q}{p-q}$ is strictly smaller than $m$ and this fact allows us to apply again Young's inequality to the last term, so that
\begin{align*}
A&\le \frac{\al(m-1)}{2\mu^p}\iint_{Q_t}|\N [|u_n|^{\mu}]|^p\,dx\,ds+\frac{1}{2m}\sup_{t\in (0,T)}\int_{\Omega}|u_n(t)|^{m}\,dx+c_2
\end{align*}
where $c_2=c_2(T,|\Omega|,m)$.\\

The $B$ term can be estimated by H\"older's and Young's inequalities with $(m,m')$ as follows:
\begin{equation}\label{B}
B\le \int_0^t\|f\|_{L^{m}(\Omega)}\|u_n\|_{L^{m}(\Omega)}^{m-1}\,ds \le \frac{1}{4m}\|u_n\|_{L^{\infty}(0,T;L^{m}(\Omega))}^{m}+c_3\|f\|_{L^1(0,T;L^{m}(\Omega))}^m.
\end{equation}
We summarize saying that, if $m\ge 2$ and $q<\frac{p}{2}$, then it holds that
\[
\begin{array}{c}
\ds
\frac{1}{4m}\sup_{t\in (0,T)}\integrale |u_n(t)|^{m}\,dx+\frac{\al(m-1)}{2\mu^p}\iint_{Q_T}|\N [|u_n|^{\mu}]|^p\,dx\,ds\\
[3mm]\ds
\le c_3\|f\|_{L^1(0,T;L^{m}(\Omega))}^m+\frac{1}{m}\integrale |u_0|^{m}\,dx+c_2.
\end{array}
\]

\noindent
\textit{The case $m\ge 2$ and $q=\frac{p}{2}$.}\\
We come back to \eqref{mq} and assume that $t\le t_1\le T$, where $t_1$ has to be fixed. Now, we have $m-1+\frac{q}{p-q}=m$ and we estimate $A$ as follows:
\begin{align*}
A&\le \frac{\al(m-1)}{2\mu^p}\iint_{Q_t}|\N [|u_n|^{\mu}]|^p\,dx\,ds+\tilde{c}_1\iint_{Q_t}|u_n|^m\,dx\,ds\\
&\le \frac{\al(m-1)}{2\mu^p}\iint_{Q_{t}}|\N [|u_n|^{\mu}]|^p\,dx\,ds+\tilde{c}_2t_1\sup_{t\in (0,t_1)}\int_{\Omega}|u_n(t)|^{m}\,dx+\tilde{c}_3.
\end{align*}
Then, letting
\begin{equation}\label{condtsub}
\tilde{c}_2t_1<\frac{3}{4m},
\end{equation}
and recalling \eqref{B}, we obtain
\[
\begin{array}{c}
\ds
\left(\frac{3}{4m}-\tilde{c}_2t_1\right)\sup_{t\in (0,t_1)}\integrale |u_n(t)|^{m}\,dx+\frac{\al(m-1)}{2\mu^p}\iint_{Q_{t_1}}|\N [|u_n|^{\mu}]|^p\,dx\,ds\\
[3mm]\ds
\le c_3\|f\|_{L^1(0,T;L^{m}(\Omega))}^m+\frac{1}{m}\integrale |u_0|^{m}\,dx+\tilde{c}_3.
\end{array}
\]
We conclude partitioning the time interval $[0,T]$ into a finite number of subintervals $[t_j,t_{j+1}]$, $0\le j\le n-1$, where $t_0=0$ and $t_n=T$ so that \eqref{condtsub} is fulfilled in each subinterval replacing $t_1$ with $t_{j+1}-t_j$. \\

We now observe that, as in Part $2$ of Theorem \ref{sapp}, the uniform boundedness of $\{u_n\}_n$ in $L^p(0,T;W_0^{1,p}(\Omega))$ can be deduced changing the function $|u_n|^{m-2}u_n$ into $\vp(u_n)=u_n$ and proceeding estimating in an analogous way.\\
We have thus proved \eqref{disappsub1}.\\

\noindent
\textit{The case $1< m<2$ and $q<\frac{p}{2}$.}\\
We deal with this case taking $\vp(u_n)=\left[(1+|u_n|)^{m-1}-1\right]\text{sign}(u_n)$ as test function. Thus, we get
\begin{equation*}
\begin{array}{c}
\ds
\integrale \Psi(u_n(t))\,dx+\al(m-1)\iint_{Q_{t}}\frac{|\N u_n|^p}{(1+|u_n|)^{2-m}}\,dx\,ds\\
[3mm]\ds
\le \gamma\iint_{Q_t}|\N u_n|^{q}(1+|u_n|)^{m-1}\,dx\,ds+\iint_{Q_t}|f|(1+|u_n|)^{m-1}\,dx\,ds+\integrale \Psi(u_0)\,dx
\end{array}
\end{equation*}
where $\Psi(v)=\int_0^v\left((1+|z|)^{m-1}-1\right)\,dz$. Taking into account the inequalities
\[
b_m|v|^m-c_m\le \Psi(v)\le a_m|v|^m+a_m,
\]
we can estimate the previous one as below
\begin{equation}\label{mq2}
\begin{array}{c}
\ds
b_m\integrale |u_n(t)|^m\,dx+\al(m-1)\iint_{Q_t}\frac{|\N u_n|^p}{(1+|u_n|)^{2-m}}\,dx\,ds\\
[3mm]\ds
\le \underbrace{\gamma\iint_{Q_t}|\N u_n|^{q}(1+|u_n|)^{m-1}\,dx\,ds}_A+\underbrace{\iint_{Q_t}|f|(1+|u_n|)^{m-1}\,dx\,ds}_B\\
[3mm]\ds
+a_m\integrale |u_0|^m\,dx+(c_m+a_m)|\Omega|.
\end{array}
\end{equation}
We first take into account the $A$ term. Then, reasoning as in the case $m\ge 2$ and $q<\frac{p}{2}$, we obtain
\begin{equation*}
\begin{split}
A\le \frac{\al(m-1)}{2}\iint_{Q_t}\frac{|\N u_n|^p}{(1+|u_n|)^{2-m}}\,dx\,ds+\frac{b_m}{2}\sup_{t\in [0,t_1]}\|u_n(t)\|_{L^m(\Omega)}^m+c_1.
\end{split}
\end{equation*}
As far as $B$ is concerned, we estimate by H\"older's inequality with $(m,m')$. Moreover, recalling \eqref{Fsub}, we have
\begin{equation}\label{B2}
\begin{split}
B&\le \|f\|_{L^1(0,T;L^m(\Omega))}\|1+|u_n|\|_{L^\infty(0,t;L^m(\Omega))}^{m-1}\\
&\le \vare\|1+|u_n|\|_{L^\infty(0,t;L^m(\Omega))}^{m}+c_\vare\|f\|_{L^1(0,T;L^m(\Omega))}^m\\
&\le \frac{b_m}{4}\sup_{t\in [0,T]}\|u_n(t)\|_{L^m(\Omega)}^m+c_2\|f\|_{L^1(0,T;L^m(\Omega))}^m+c_3
\end{split}
\end{equation}
where the intermediate passage is due to Young's inequality with $(m,m')$.\\ 
Thus, we obtain
\begin{equation*}
\begin{array}{c}
\ds
\frac{b_m}{4}\sup_{t\in [0,T]}\|u_n(t)\|_{L^m(\Omega)}^m+\frac{\al(m-1)}{2}\iint_{Q_{t_1}}\frac{|\N u_n|^p}{(1+|u_n|)^{2-m}}\,dx\,ds\\
[3mm]\ds
\le a_m\integrale |u_0|^m\,dx+c_2\|f\|_{L^1(0,T;L^m(\Omega))}^m+C
\end{array}
\end{equation*}
where $C=C(T,m,|\Omega|)$.\\

\noindent
\textit{The case $1< m<2$ and $q=\frac{p}{2}$.}\\
We proceed as before setting $t\le t_1\le T$, where $t_1$ has to be defined. Then, taking into account the $A$ term in \eqref{mq2}, we apply Young's inequality with indices $(2,2)$ so we get
\begin{equation*}
\begin{split}
A&\le \frac{\al(m-1)}{2}\iint_{Q_t}\frac{|\N u_n|^p}{(1+|u_n|)^{2-m}}\,dx\,ds+\tilde{c}_1\iint_{Q_t}(1+|u_n|)^m\,dx\,ds\\
&\le \frac{\al(m-1)}{2}\iint_{Q_t}\frac{|\N u_n|^p}{(1+|u_n|)^{2-m}}\,dx\,ds+\tilde{c}_2t_1\sup_{t\in [0,t_1]}\|u_n(t)\|_{L^m(\Omega)}^m+\tilde{c}_3T.
\end{split}
\end{equation*}
Setting $t_1$ so that 
\begin{equation}\label{condt}
\frac{3b_m}{4}-\tilde{c}_2t_1>0
\end{equation}
and recalling \eqref{B2}, we deduce 
\begin{equation*}
\begin{array}{c}
\ds
\left(\frac{3b_m}{4}-\tilde{c}_2t_1\right)\sup_{t\in [0,t_1]}\|u_n(t)\|_{L^m(\Omega)}^m+\frac{\al(m-1)}{2}\iint_{Q_{t_1}}\frac{|\N u_n|^p}{(1+|u_n|)^{2-m}}\,dx\,ds\\
[3mm]\ds
\le a_m\integrale |u_0|^m\,dx+c_2\|f\|_{L^1(0,T;L^m(\Omega))}^m+C
\end{array}
\end{equation*}
where $C=C(T,m,|\Omega|)$.\\
We conclude partitioning the time interval $[0,T]$ into a finite number of subintervals $[t_j,t_{j+1}]$, $0\le j\le n-1$, where $t_0=0$ and $t_n=T$ so that \eqref{condt} is fulfilled in each subinterval. We have thus proved \eqref{disappsub1}.\\
 
\noindent
\textit{The case $m=1$ and $q<\frac{p}{2}$.}\\
We claim that we are within the assumptions of Lemma \ref{MC}. Indeed, we can reason as in the first part of Theorem \ref{sapp3} and obtain the $L^1$ data classical estimates (see \cite{BG}):
\[
\sup_{t\in [0,T]}\integrale |u_n(t)|\,dx\le M
\]
and 
\[
\al\iint_{Q_T}|\N T_k(u_n)|^p\,dx\,ds\le kM
\]
where 
\[
M=\gamma\iint_{Q_T}|\N u_n|^q\,dx\,ds+\iint_{Q_T}|f|\,dx\,ds+\integrale|u_0|\,dx.
\]
Then, Lemma \ref{MC} implies that
\[
\||\N u_n|^{q}\|_{M^{\frac{p}{2q}}(Q_T)}^{\frac{p}{2q}}\le c\left[ \||\N u_n|^q\|_{L^1(Q_T)}+\|f\|_{L^1(Q_T)}+\|u_0\|_{L^1(\Omega)}
\right]
\]
from which, being $q<\frac{p}{2}$, we obtain
\[
\||\N u_n|^{q}\|_{L^{1}(Q_T)}\le c \||\N u_n|^q\|_{L^1(Q_T)}^{\frac{2q}{p}}+c_0
\]
where $c_0=c_0(\|f\|_{L^1(Q_T)},\|u_0\|_{L^1(\Omega)})$ above the parameters of the problem. An application of Young's inequality and the assumption $q<\frac{p}{2}$ give us
\[
\||\N u_n|^{q}\|_{L^{1}(Q_T)}\le c.
\]
The uniform boundedness of $\{|\N u_n|^q\}_n$ in $L^1(Q_T)$ allows us to conclude saying that $\{u_n\}_n$ satisfies the estimates in Lemmas \ref{MGN} and \ref{MC}.
\end{proof}

\subsection{Some convergence results}
\begin{proposition}\label{subae}
Assume $1<p<2$, \eqref{A1}, \eqref{A2}, \eqref{A3}, \eqref{Hsub},  either \eqref{id1} or \eqref{id2}, \eqref{Fsub} and let $\{u_n\}_n$ be a sequence of solutions of \eqref{Pn}. Then, we have that, for some function $u$,
\[
u_n\to u \quad \text{a.e. in }\,\,Q_T,
\] 
\[
\N u_n\to\N u\quad \text{a.e. in }\,\,Q_T,
\] 
\[
H_n(t,x,\nabla u_n)\to H(t,x,\nabla u) \quad\text{strongly in}\quad L^1(Q_T).
\]
In particular, if either $1<m<2$ and $q\le\frac{p}{2}$ or $m=1$ and $q<\frac{p}{2}$ then
\[
T_k(u_n)\to T_k(u)\quad\text{strongly in}\quad L^p(0,T;W^{1,p}_0(\Omega))\,\,\forall k>0.
\]
\end{proposition}
\begin{proof}
We just deal with the case $1\le m<2$, since having $m\ge 2$ allows us to reason as in  Corollary \ref{simon} and Proposition \ref{aeD}.\\

We start proving that $\{|\N u_n|^b\}_n$ is bounded in $L^1(Q_T)$ for some $q<b<p$.\\

\noindent
\textit{The case $1<m<2$ and $\frac{2N}{N+m}<p<2$.}\\
Theorem \ref{sappsub} provides that $\{(1+|u_n|)^{\mu-1}u_n\}_n\subseteq L^\infty(0,T;L^{\frac{m}{\mu}}(\Omega))\cap  L^p(0,T;W_0^{1,p}(\Omega))$. Note that $\frac{m}{\mu}<p^*$ if and only if $p>\frac{2N}{N+m}$: we thus reason as in Proposition \ref{abc}, getting $\ds b=p\frac{N\mu+m}{ N+m}$. Algebraic computations show that such a value of $b$ satisfies the inequalities $\ds q\le\frac{p}{2}<b$, being $m> 1$, and $\ds b<p$, since $\mu<1$. \\

\noindent
\textit{The case $1<m<2$ and $1<p\le\frac{2N}{N+m}$.}\\
Again, we reason as in Proposition \ref{abc} estimating the $b$ power of the gradient as
\begin{equation*}
\iint_{Q_T}|\nabla u_n|^b\,dx\,ds\le \biggl(\iint_{Q_T} \frac{|\nabla u_n|^p}{(1+|u_n|)^{p(1-\mu)}}\,dx\,dt\biggr)^{\frac{b}{p}}\biggl(\iint_{Q_T} (1+|u_n|)^{pb\frac{(1-\mu)}{p-b}}\,dx\,dt\biggr)^{\frac{p-b}{p}}.
\end{equation*}
However, since our current range of $p$ implies $\frac{m}{\mu}>p^*$, we are forced to require $\ds pb\frac{(1-\mu)}{p-b}=m$. This condition leads to $\ds b=\frac{p}{2}m$ which verifies $q< \frac{p}{2}m<p$ since $1<m<2$. \\

Since the r.h.s. is bounded in $L^1(Q_T)$, we can prove the a.e. convergences of $u_n\to u$ and $\N u_n\to \N u$ as in Proposition \ref{abc}. More precisely, we have to split the proof between $b$ greater and smaller than one and follows the line of the Proposition quoted above.
\\

Now, having $q<b$, we deduce the equi-integrability of the r.h.s. and the strong convergence in $L^1(Q_T)$ of the r.h.s. follows, again, by Vitali Theorem.
%

The strong convergence of the truncation function can be deduced as in Proposition \ref{abc}.\\

\noindent
\textit{The cases $m=1$.}\\
We just recall Proposition \ref{abc3} and say that we can proceed in the same way.
\end{proof}

\subsection{The existence result}
\begin{theorem}\label{Tsub}
Let $1<p<2$ and assume \eqref{A1}, \eqref{A2}, \eqref{A3}, \eqref{Hsub},  either \eqref{id1} or \eqref{id2} and \eqref{Fsub}. 
Then, there exists at least one weak solution of the problem \eqref{Psub} such that
\begin{itemize}
\item if $m\ge 2$, then
\[
u\in L^p(0,T; W^{1,p}_0(\Omega))
\]
and satisfies the weak formulation
\[
-\integrale u_0(x)\vp(0,x)\,dx+\iint_{Q_T}-u\vp_t+a(t,x,u,\N u)\cdot\N\vp\,dx\,dt=\iint_{Q_T}H(t,x,\N u)\vp\,dx\,dt
\]
for every test function $\vp\in L^p(0,T;W_0^{1,p}(\Omega))\cap  L^{\infty}(Q_T)$ such that $\vp_t\in L^{p'}(Q_T)$ and $\vp(T,x)=0$. Moreover, this solution fulfils the following regularity:
\[
|u|^\mu\in L^p(0,T;W^{1,p}_0(\Omega))\quad\text{with}\quad \mu=\frac{m-2+p}{p};
\] 
\item if $1<m<2$, then
\[
u\in \mathcal{T}_0^{1,p}(Q_T), 
\]
\[
H(t,x,\nabla u)\in L^1( Q_T),
\]
\[
\begin{array}{c}
\ds
-\int_\Omega S(u_0)\varphi(0)\,dx+\iint_{Q_T} \left[-S(u)\varphi_t+ a(t,x,u,\N u)\cdot \N(S'(u)\vp)\right]\,dx\,dt\\
[3mm]\ds
= \iint_{Q_T} H(t,x,\nabla u)S'(u)\varphi\,dx\,dt,
\end{array}
\]
for every $S\in W^{2,\infty}(\mathbb{R})$ such that $S'(\cdot)$ has compact support and for every test function $\vp\in L^p(0,T;W_0^{1,p}(\Omega))\cap  L^{\infty}(Q_T)$ such that $\vp_t\in L^{p'}(Q_T)$ and $\vp(T,x)=0$. Moreover, such a solution fulfils the following regularities:
\begin{equation*}
(1+|u|)^{\mu-1}u\in L^p(0,T;W_0^{1,p}(\Omega))\quad\text{with}\quad \mu=\frac{m-2+p}{p};
\end{equation*}
\item if $m=1$, then
\[
u\in \mathcal{T}_0^{1,p}(Q_T), 
\]
\[
 H(t,x,\nabla u)\in L^1( Q_T),
\]
\[
\lim_{n\to \infty}\frac{1}{n}\iint_{\{n\le |u|\le 2n\}}a(t, x,u,\N u)\cdot \N u=0, 
\]
\[
\begin{array}{c}
\ds
-\int_\Omega S(u_0)\varphi(0)\,dx+\iint_{Q_T} \left[-S(u)\varphi_t+ a(t,x,u,\N u)\cdot \N(S'(u)\vp)\right]\,dx\,dt\\
[3mm]\ds
= \iint_{Q_T} H(t,x,\nabla u)S'(u)\varphi\,dx\,dt,
\end{array}
\]
for every $S\in W^{2,\infty}(\mathbb{R})$ such that $S'(\cdot)$ has compact support and for every test function $\vp\in L^p(0,T;W_0^{1,p}(\Omega))\cap  L^{\infty}(Q_T)$ such that $\vp_t\in L^{p'}(Q_T)$ and $\vp(T,x)=0$.
\end{itemize}
\end{theorem}
\begin{proof}
The a priori estimate proved in Theorem \ref{sappsub} and the convergence results in Proposition \ref{subae} allow us to reason as in Theorem \ref{Tfe} if $m\ge 2$, Theorem \ref{Trin} whenever $1<m<2$ and, finally, as in Theorem \ref{Trin2} as $m=1$.
\end{proof}
\begin{remark}
In particular, $|\N u|^b\in L^1(Q_T)$ for $b=\frac{p}{2}m$ if $1<p\le \frac{2N}{N+m}$ and $b=\frac{N(m-2+p)+pm}{N+m}$ if $\frac{2N}{N+m}<p<N$. 
\end{remark}


\appendix
\section*{Appendices}\label{app}
\renewcommand{\thesubsection}{\Alph{subsection}}
\subsection{The approximating problem and some preliminary results}\label{secpbn}
\setcounter{equation}{0}
\renewcommand{\theequation}{a.\arabic{equation}}

\begin{theoa}[Gagliardo-Nirenberg inequality]\label{teoGN}
Let $\Omega\subset \mathbb{R}^N$ be a bounded and open subset and $T$ a real positive number. Then, if
\begin{equation}\label{GN}
v\in L^{\infty}(0,T;L^{h}(\Omega))\cap L^{\eta}(0,T;W_0^{1,\eta}(\Omega))
\end{equation}
where
\[
1\le \eta<N \quad \text{and}\quad 1\le h\le\eta^*,
\]
we have that
\begin{equation*}
v\in L^y(0,T;L^w(\Omega))
\end{equation*}
where the couple $(w,y)$ fulfils
\[
h\le w\le \eta^*,\,\,\eta\le y\le \infty
\]
and satisfies the relation
\begin{equation}\label{rel}
\frac{Nh}{w}+\frac{ N(\eta-h)+\eta h }{y}=N.
\end{equation}
Moreover, the following inequality holds:
\begin{equation}\label{disGN}
\int_0^T \|v(t)\|_{L^w(\Omega)}^y\,dt\le c(N,\eta,h)\|v\|_{L^{\infty}(0,T;L^h(\Omega))}^{y-\eta}\int_0^T\|\N v(t)\|_{L^{\eta}(\Omega)}^{\eta}\,dt.
\end{equation}
In particular, having $w=y$ implies that 
\begin{equation*}
v\in L^{\eta\frac{N+h}{N}}(Q_T)
\end{equation*}
and the estimate reads
\begin{equation}\label{disGN=}
\int_0^T \|v(t)\|_{L^w(\Omega)}^w\,dt\le c(N,\eta,h)\|v\|_{L^{\infty}(0,T;L^h(\Omega))}^{\frac{\eta h}{N}}\int_0^T\|\N v(t)\|_{L^{\eta}(\Omega)}^{\eta}\,dt.
\end{equation}
\end{theoa}

The approximating problem we consider during the paper is the following:
\begin{equation}\label{Pn}\tag{$P_n$}
\begin{cases}
\begin{split}
& (u_n)_t-\dive a(t,x,u_n,\nabla u_n)=H_n(t,x,\nabla u_n)\qquad & Q_T,\\
& u_n=0 \qquad &(0,T)\times \partial \Omega,\\
& u_n(0,x)=u_{0,n}(x)\qquad & \Omega,
\end{split}
\end{cases}
\end{equation}
where $\{u_{0,n}\}_n=\{T_n(u_0)\}_n\subseteq L^\infty(\Omega)$ and $\{H_n(t,x,\xi)\}_n=\{T_n(H(t,x,\xi))\}_n$. Thanks to \cite{G} (see also \cite{BMP} if $p=2$ and \cite{DGP}, \cite{OP}) we have that \eqref{Pn} admits (at least) a solution $u_n$ such that
\[
u_n\in L^p(0,T;W^{1,p}_0(\Omega))\cap L^{\infty}(Q_T),\,\, (u_n)_t \in L^{p'}(0,T;W^{-1,p'}(\Omega))+L^1(Q_T),
\]
\[
-\int_{\Omega}u_{0,n}\varphi(0,x)\,dx+\iint_{Q_T}\bigl[-\varphi_t u_n+a(t,x,u_n,\nabla u_n)\cdot\nabla \varphi\bigr]  \,dt\,dx=\iint_{Q_T}  H_n(t,x,\nabla u_n)\varphi \,dt\,dx,
\]
for every $\varphi\in C_c^{\infty}([0,T)\times\Omega)$  and for every fixed $n\in \mathbb{N}$.\\

Since we will take more general test functions, we here recall a useful result aimed at justifying our choices. Its proof is a consequence of \cite[Lemma $4.6$ and Corollary $4.7$]{PPP}.

\begin{propa}
Let $u_n$ be a solution of \eqref{Pn}.
Then
\begin{multline*}
\int_{\Omega}\Psi(u_n(t))\,dx+\iint_{Q_t}a(s,x,u_n,\nabla u_n)\cdot \nabla u_n\psi'(u_n)\,dx\,ds\\=\iint_{Q_t} H(s,x,\nabla u_n)\psi(u_n) \,dx\,ds+\int_{\Omega}\Psi(u_{0,n})\,dx \qquad \text{a.e.}\quad t\in(0,T),
\end{multline*}
for every $\psi\in W^{1,\infty}(\mathbb{R})$ such that $\psi(0)=0$, where $\Psi(v)=\int_0^v \psi (w)\,dw$.
\end{propa}

\subsection{Marcinkievicz Lemmas}\label{marc}
\setcounter{equation}{0}
\renewcommand{\theequation}{b.\arabic{equation}}
Let $v$ a measurable real function belonging to $L^\infty(0,T;L^1(\Omega))$ and satisfying certain growth assumption on the $L^p(Q_T)$ norm of $|\N T_k(v)|$. We are going to prove two Lemmas which provide estimates on the measures of the level sets of suitable powers of $v$ and $|\N v|$.\\
For further result in this sense, we refer to \cite[Appendix $A$]{DNFG} as far as the parabolic setting is concerned and to \cite[Appendix $A$]{BMMP}, \cite[Section $4$]{BBGGPV} regarding the stationary problem.

\begin{lemb}\label{MGN}
Let $1<p<N$, $v\in \mathcal{T}^{1,p}_0(Q_T)$ be such that
\begin{equation*}
\|v\|_{L^\infty(0,T;L^1(\Omega))}\le M
\end{equation*}
and
\begin{equation*}
\int_0^T\|\N T_k(v(t))\|_{L^p(\Omega)}^p\,dt\le  Mk
\end{equation*}
for every $k>0$. Then, we have that
$|v|^{\frac{p(N+1)-N}{N+p}}\in M^{\frac{N+p}{N}}(Q_T)$ and
$|\N v|^{\frac{p(N+1)-N}{N+2}}\in M^{\frac{N+2}{N+1}}(Q_T)$. Moreover, the following estimates hold:
\begin{equation}\label{loru1}
\||v|^{\frac{p(N+1)-N}{N+p}}\|_{M^{\frac{N+p}{N}}(Q_T)}\le c M
\end{equation}
and
\begin{equation}\label{lorDu1}
\||\N v|^{\frac{p(N+1)-N}{N+2}}\|_{M^{\frac{N+2}{N+1}}(Q_T)}\le c  M
\end{equation}
where the constants $c$ depend on $N$, $p$ and $q$.
\end{lemb}
\begin{proof}
Let us begin with the estimate concerning $v$. \\
Gagliardo-Nirenberg regularity results (see Theorem \ref{teoGN} when $w=y$) provide that $T_k(v)$ satisfies the following inequality:
\[
\int_0^T \|T_k(v(t))\|_{L^\eta(\Omega)}^\eta\,dt\le c\|T_k(v)\|_{L^\infty(0,T;L^1(\Omega))}^{\eta-p}\int_0^T\|\N T_k(v(t))\|_{L^p(\Omega)}^p\,dt
\] 
where $\eta=p\frac{N+1}{N}$. Thus, the bound above gives us the following estimate:
\[
\int_0^T \|T_k(v(t))\|_{L^\eta(\Omega)}^\eta\,dt\le cM^{\eta-p+1}k.
\]
This means that
\[
\begin{array}{c}
\ds
\meas \{(t,x)\in Q_T:\,\, |v|>k \}\le \frac{\int_0^T \|T_k(v(t))\|_{L^\eta(\Omega)}^\eta\,dt}{k^\eta}\le cM^{\eta-p+1}k^{-(\eta-1)}\\
[3mm]
=cM^{\frac{p+N}{N}} k^{-\frac{p(N+1)-N}{N}}.
\end{array}
\]
Taking $k=h^{\frac{N+p}{p(N+1)-N}}$, we obtain
\[
h\,\meas \{(t,x)\in Q_T:\,\, |v|^{\frac{p(N+1)-N}{N+p}}>h \}^{\frac{N}{N+p}}\le cM
\]
and so the first part is concluded.

We go further defining the function
\[
\vp(k,\lm)=\meas\{(t,x)\in Q_T:\,\, |\N v|^p>\lm\,\text{ and }\, |v|>k \}
\]
and observing that, being $\vp(k,\cdot)$ non increasing in the $\lm$ variable, the following inequalities hold:
\begin{align*}
\vp(0,\lm)\le \frac{1}{\lm}\int_0^\lm \vp(0,\theta)\,d\theta=\frac{1}{\lm}\int_0^\lm 
\vp(k,\theta)+[\vp(0,\theta)-\vp(k,\theta)]
\,d\theta\le \vp(k,0)+\frac{1}{\lm}\int_0^\lm 
\vp(0,\theta)-\vp(k,\theta)
\,d\theta.
\end{align*}
Moreover, since  
\[
\vp(0,\theta)-\vp(k,\theta)=\meas\{ (t,x)\in Q_T:\,\, |\N v|^{p}>\theta\,\text{ and }\, |v|\le k  \}
\]
and thus 
\[
\int_0^\infty 
[\vp(0,\theta)-\vp(k,\theta)]
\,d\theta=\int_0^T\|\N T_k(v(t))\|_{L^p(\Omega)}^p\,dt\le k M
\]
we finally get
\begin{align*}
\vp(0,\lm)\le \vp(k,0)+\frac{Mk}{\lm}.
\end{align*}
The definition of $\vp(k,\lm)$ allows us to say that
\[
\meas \{(t,x)\in Q_T:\,\, |\N v|^{p}>\lm \}\le cM^{\frac{p+N}{N}} k^{-\frac{p(N+1)-N}{N}}+\frac{Mk}{\lm}
\]
which, minimizing the r.h.s. in $k$, becomes
\[
\meas \{(t,x)\in Q_T:\,\, |\N v|^{p}>\lm \}\le cM^{\frac{N+2}{N+1}}\lm^{-\frac{p(N+1)-N}{p(N+1)}}.
\]
The estimate \eqref{lorDu1} follows taking $\lm=h^{p\frac{N+2}{p(N+1)-N}}$.\\
\end{proof}

\begin{rmkappb}
Note that we have just proved that the assumptions given in the previous statement ensure $v\in M^{\frac{p(N+1)-N}{N}}(Q_T)$ and $|\N v|\in M^{\frac{p(N+1)-N}{N+1}}(Q_T)$ which are the regularities satisfied by the heat problem (we refer to \cite{AMST,ST}).
\end{rmkappb}

\begin{lemb}\label{MC}
Let $1<p<N$, $v\in \mathcal{T}^{1,p}_0(Q_T)$ be such that
\begin{equation*}
\|v\|_{L^\infty(0,T;L^1(\Omega))}\le M
\end{equation*}
and
\begin{equation*}
\int_0^T\|\N T_k(v(t))\|_{L^p(\Omega)}^p\,dt\le  Mk
\end{equation*}
for every $k>0$. Then, we have that
$|\N v|^{\frac{p}{2}}\in M^1(Q_T)$. Moreover, the following estimate holds:
\begin{equation}\label{lorDu2}
\||\N v|^{\frac{p}{2}}\|_{M^{1}(Q_T)}\le c  M.
\end{equation}
where the constant $c$ depends on $N$, $p$ and $q$.
\end{lemb}
\begin{proof}
We just proceed as before changing the Gagliardo-Nirenberg inequality into \v{C}eby\v{s}\"ev's one, so our starting step reads
\[
\meas \{(t,x)\in Q_T:\,\, |v|>k \}\le c\frac{M}{k}.
\]
\end{proof}
\begin{rmkappb}[Comparison between estimates \eqref{lorDu1} and \eqref{lorDu2}]\label{rmkhom}
We point out that \eqref{lorDu1} and \eqref{lorDu2} respectively imply
\begin{equation*}
\begin{array}{c}
\ds
\|\N v\|_{M^{\frac{p(N+1)-N}{N+1}}(Q_T)}\le cM^{\frac{N+2}{p(N+1)-N}}\\
[3mm]
\ds
\|\N v\|_{M^\frac{p}{2}(Q_T)}\le cM^{\frac{2}{p}}
\end{array}
\end{equation*}
which hold \emph{for every $1<p<N$}. However, if we focus on the \emph{regularity} we have that \eqref{lorDu1} is better than \eqref{lorDu2} when $p>\frac{2N}{N+1}$. If instead we take into account the \emph{homogeneity exponent}, we have that \eqref{lorDu1} exhibits a preferable bound than \eqref{lorDu2} only when $p\ge 2$.\\
\end{rmkappb}

\subsection{On the sharpness of the assumptions}\label{sharp}
\setcounter{equation}{0}
\renewcommand{\theequation}{c.\arabic{equation}}

We go on with our analysis observing that the assumption \eqref{ID1} is the weakest one, within the class of Lebesgue spaces, which allows us to have an existence result. Roughly speaking, assuming \eqref{A1}, \eqref{A2}, \eqref{A3}, \eqref{H} and $u_0$ in a Lebesgue space $L^{\eta}(\Omega)$ with $1\le\eta<\si$ \emph{does not allow} \eqref{P} \emph{to necessarily admit a solution} (in some sense).\\

For the sake of simplicity, let us consider the problem \eqref{P} with the $p$-Laplace operator, the $q$ power of the gradient in the r.h.s. and zero forcing term $f$:
\begin{equation}\label{D}
\begin{cases}
\begin{array}{ll}
u_t-\D_p u=\gamma |\N u|^q &\text{in}\,\,Q_T,\\
u=0 &\text{on}\,\,(0,T)\times\partial\Omega,\\
u(0,x)=u_0(x)  & \text{in}\,\, \Omega
\end{array}
\end{cases}
\end{equation}
for some positive $\ga$.\\
Our goal is showing that there exists some initial datum
\begin{equation}\label{IDNE} 
u_0\in L^{\eta}(\Omega),\quad 1\le\eta<\si,
\end{equation}
such that the problem \eqref{D} does not admit any solution $u$ such that 
\begin{equation}\label{solne1}
\begin{array}{c}
\ds
u\in L_{loc}^p(0,T;W_0^{1,p}(\Omega))\cap C([0,T];L^{\eta}(\Omega)),\\
[4mm]\ds
|u|^\frac{\eta+p-2}{p}\in  L^p(0,T;W_0^{1,p}(\Omega)).
\end{array}
\end{equation}
Note that we are no longer dealing with functions belonging to the set
\begin{equation*}
\left\{u\,\,\text{solving \eqref{D}}: \quad |u|^{\frac{\si+p-2}{p}}\in L^p(0,T;W_0^{1,p})
\right\}.
\end{equation*}
since $\frac{\eta+p-2}{p}<\frac{\si+p-2}{p}$. \\
 
With this purpose, we follow the lines of \cite[Subsection $3.2$]{BASW}, that is, we prove suitable integral inequalities for $u$ (from above) and for  $U$ (from below), where $U=U(t,x)$ is the solution of the $p$-Laplace Cauchy-Dirichlet problem
\begin{equation}\label{UU}
\begin{cases}
\begin{array}{ll}
U_t-\D_p U=0 &\text{in}\,\,Q_T,\\
U=0 &\text{on}\,\,(0,T)\times\partial\Omega,\\
U(0,x)=u_0(x)  & \text{in}\,\, \Omega.
\end{array}
\end{cases}
\end{equation}
Having $u=U=0$ on $(0,T)\times\partial\Omega$ and also $u_0(x)=U(0,x)$ in $\Omega$ allows us to apply standard comparison results for $u_t-\D_p u$ (we quote, for instance, \cite{Di}) and deduce that $U\le u$.

\begin{theoc}
Let us consider \eqref{D} with $p\ge 2$ and $u_0(x)=|x|^{-\frac{N}{\eta}+\om}\chi_{\{|x|<1\}}$ for $\om>0$ sufficiently small, so $u_0$ fulfils \eqref{IDNE}. Then, \eqref{D} does not admit any solution verifying \eqref{solne1}.
\end{theoc}
\begin{proof}
\textit{Lower bound for $U$.}\\
We start recalling the Harnack inequality (see \cite[Chapter $VI$ Paragraph $8$]{Di}) satisfied by $U(t,x)$:
\begin{equation*}
\fint_{B_r}u_0(x)\,dx\le c\left\{
\left(\frac{r^p}{t}\right)^\frac{1}{p-2}+
\left(\frac{t}{r^p}\right)^{\frac{N}{p}}\left[
\inf_{y\in B_r}U(t,y)\right]^{\frac{\lm}{p}}
\right\}
\end{equation*}
where $t>0$, $B_r=B_r(0)$, $c=c(N,p)$ and $\lm =p(N+1)-2N$.\\
The particular choice of $u_0$ implies that
\[
\fint_{B_r}u_0(x)\,dx\ge cr^{-\frac{N}{\eta}+\om}
\]
and so, taking $r$ such that 
\begin{equation}\label{start}
t\gg cr^{p+\frac{(N-\om \eta)(p-2)}{\eta}},
\end{equation}
we have $r^{-\frac{N}{\eta}+\om}\gg \left(\frac{r^p}{t}\right)^\frac{1}{p-2}$. Then, we are allowed to say that
\begin{equation*}
\inf_{y\in B_r}U(t,y)\ge c t^{-\frac{N}{\lm}}r^{\frac{pN}{\lm}\frac{\eta-1}{\eta} +p\frac{\om}{\lm}}
\end{equation*}
and also to deduce
\begin{equation}\label{infU}
\int_{B_r}U(t,y)\,dy\ge c t^{-\frac{N}{\lm}}r^{\frac{pN}{\lm}\frac{\eta-1}{\eta} +p\frac{\om}{\lm}+N}.
\end{equation}
\noindent
\textit{Upper bound for $u$.}\\
We now look for a bound from above for the integral in the space variable of the solution of \eqref{D}.
First, we observe that such a solution $u$ should fulfil:
\begin{equation}\label{solne2}
|u|^{\frac{\eta+q-1}{q}}\in L^q(0,T;W^{1,q}_0(\Omega))
\end{equation}
where this last regularity follows from the boundedness of
\[
\iint_{Q_T} |\N u|^q|u|^{\eta-1}\,dx\,dt<\infty.
\]
Note that the above boundedness holds thanks to \eqref{solne1} and follows reasoning as in \eqref{3Hpp}.
Then, there exists at least a sequence $\{t_j\}_j$ satisfying $t_j\to 0$ such that, applying also H\"older's inequality with indices $\left(q^*\frac{\eta+q-1}{q}, \left(q^*\frac{\eta+q-1}{q} \right)'\right)$, we have
\begin{equation*} 
\int_{\Omega}|\N u(t_j)|^q|u(t_j)|^{\eta-1}\,dx=\|\N (|u(t_j)|^{\frac{\eta+q-1}{q}})\|_{L^q(\Omega)}^q\le \frac{1}{t_j}. 
\end{equation*}
Then, we obtain
\begin{equation}\label{u}
\begin{split}
\int_{B_r}u(t_j,y)\,dy &\le \|u(t_j)\|_{L^{q^*\frac{\eta+q-1}{q}}(\Omega)}
r^{N-\frac{Nq}{q^*(\eta+q-1)}}\\
&\le c\|\N (|u(t_j)|^{\frac{\eta+q-1}{q}})\|_{L^q(\Omega)}^{\frac{q}{\eta+q-1}}\,
r^{N-\frac{Nq}{q^*(\eta+q-1)}}\\
& \le ct_j^{-\frac{1}{\eta+q-1}}r^{N-\frac{N-q}{\eta+q-1}}.
\end{split}
\end{equation}

\noindent
\textit{Conclusion.}\\
We already know that $U\le u$ where $u$ and $U$ are, respectively, solutions of \eqref{D} and \eqref{UU}. Then we take advantage of this information gathering \eqref{infU} and \eqref{u}, so we get
\begin{equation*}
t_j^{-\frac{N}{\lm}}r^{N+\frac{pN}{\lm}\frac{\eta-1}{\eta} +p\frac{\om}{\lm}}\le\int_{B_r}U(t_j,y)\,dy\le\int_{B_r}u(t_j,y)\,dy\le ct_j^{-\frac{1}{\eta+q-1}}r^{N-\frac{N-q}{\eta+q-1}}
\end{equation*}
from which
\begin{equation}\label{dis}
t_j^{\frac{1}{\eta+q-1}}r^{\frac{pN}{\lm}\frac{\eta-1}{\eta} +p\frac{\om}{\lm}+\frac{N-q}{\eta+q-1}-\frac{N}{\lm}}\le c.
\end{equation}
We recall \eqref{start} and set $r=\omega t_j^{\frac{\eta}{p\eta+N(p-2)-\om\eta(p-2)}}$, where $0<\omega\ll 1$, obtaining
\begin{equation}\label{dis222}
t_j^{\vp}\le c(\omega)
\end{equation}
for
$\vp=\vp(\eta)$ defined by
\begin{align*}
\vp=-\frac{N}{\lm}+\frac{1}{\eta+q-1}+
\frac{p}{\lm}\frac{\eta(N+\om)- N}{\eta(p-\om(p-2))+N(p-2)}+\frac{\eta(N-q)}{(\eta+q-1)(\eta(p-\om(p-2))+N(p-2))}.
\end{align*}
This means that, as $j\to\infty$, we need to have $\vp\ge 0$ in order to have \eqref{dis222} fulfilled. Algebraic computations lead us to the equivalent request
\[
\begin{array}{c}
\ds
(\eta+q-1)\left[
-N\eta(p-\om(p-2))-N^2(p-2)+p\eta(N+\om)-Np
\right]\\
[4mm]
\ds
+\lm[(p-\om(p-2))\eta+N(p-2)+\eta(N-q)]\\
[4mm]\ds
=(\eta+q-1)\left[
\eta\om(N(p-2)+p)-N(N(p-2)+p)
\right]\\
[4mm]
\ds
+\lm[\eta(N+p-q-\om(p-2))+N(p-2)]\\
[4mm]\ds
=-\lm(\eta+q-1)(N-\om\eta )
+\lm[\eta(N+p-q-\om(p-2))+N(p-2)]\ge 0
\end{array}
\]
by the definition of $\lm$. Then, looking for $\vp\ge 0$, we erase $\lm$ getting
\[
\begin{array}{c}
\ds
\eta(p-q)-N(q-(p-1))-\om\eta(\eta+q-p+1) \ge 0
\end{array}
\]
from which we deduce
\begin{equation}\label{sieta}
\si-\eta\le \frac{\om\eta}{p-q}(\eta+q-p+1)
\end{equation}
thanks to the definition of $\si$.\\
Since we can choose $\om$ sufficiently small such that \eqref{sieta} is violated, then we deduce the assertion by contradiction.
\end{proof}

\section*{Acknowledgements}
The author wishes to thank Andrea Dall'Aglio and Alessio Porretta for their advices, comments and support during the composition of this work.

\end{document}